\documentclass[11pt]{amsart}

\usepackage{geometry}
\usepackage{amsfonts, amssymb, amscd}
\numberwithin{equation}{section}

\usepackage[symbol]{footmisc}

\usepackage{bm}

\usepackage{verbatim}
\usepackage{amssymb}
\usepackage{mathrsfs}
\usepackage{graphicx}
\usepackage[nospace,noadjust]{cite}
\usepackage[all]{xy}
\usepackage{tikz-cd}
\usepackage{subcaption}
\usepackage{subfiles}
\usepackage[toc,page]{appendix}
\usepackage{comment}
\usepackage{enumerate}
\usepackage{enumitem}

\usepackage{graphicx}
\graphicspath{{images/}}

\usepackage{appendix}
\usepackage{hyperref}
\usepackage{romannum}
\AtBeginDocument{\pagenumbering{arabic}}

\newcommand{\dashto}{\dashrightarrow}

\newcommand{\rddown}[1]{\left\lfloor{#1}\right\rfloor} % round _down

\newcommand{\Oo}{\mathcal{O}}

\newcommand{\Supp}{\operatorname{Supp}}
\newcommand{\mult}{\operatorname{mult}}
\newcommand{\Center}{\operatorname{Center}}

\newcommand{\Sing}{\operatorname{Sing}}

\newcommand{\Cc}{\mathbb{C}}

\newcommand{\Pp}{\mathbb{P}}
\newcommand{\Qq}{\mathbb{Q}}

\newcommand{\Rr}{\mathbb{R}}
\newcommand{\F}{\mathbb{F}}

\newcommand{\fR}{\mathfrak{R}}

\newcommand{\Zz}{\mathbb{Z}}

\newcommand{\cN}{\mathcal{N}}

\newcommand{\cZ}{\mathcal{Z}}

\newcommand{\bM}{\mathbf{M}}

\newcommand{\WDiv}{\operatorname{WDiv}}

\newcommand{\lf}{\lfloor}
\newcommand{\rf}{\rfloor}

\newcommand{\Ii}{{\Gamma}}

 \usepackage{todonotes}

\newtheorem{thm}{Theorem}[section]
\newtheorem{conj}[thm]{Conjecture}

\newtheorem{lem}[thm]{Lemma}

\newtheorem{prop}[thm]{Proposition}

\newtheorem{claim}[thm]{Claim}
\theoremstyle{definition}
\newtheorem{defn}[thm]{Definition}
\newtheorem{rem}[thm]{Remark}

\newtheorem{ex}[thm]{Example}

\theoremstyle{definition}

\usepackage{todonotes}
\begin{document}

\title{Boundedness of complements for log Calabi-Yau threefolds}

\author{Guodu Chen}

\address{Institute for Theoretical Sciences, Westlake University, Hangzhou, Zhejiang, 310024, China}
\email{chenguodu@westlake.edu.cn}

\author{Jingjun Han}
\address{Shanghai Center for Mathematical Sciences, Fudan University, Shanghai 200438, China}
\email{hanjingjun@fudan.edu.cn}

\author{Qingyuan Xue}
\address{Department of Mathematics, The University of Utah, Salt Lake City, UT 84112, USA}
\email{xue@math.utah.edu}

\begin{abstract}
In this paper, we study the theory of complements, introduced by Shokurov, for Calabi-Yau type varieties with the coefficient set $[0,1]$. We show that there exists a finite set of positive integers $\mathcal{N}$, such that if a threefold pair $(X/Z\ni z,B)$ has an $\mathbb{R}$-complement which is klt over a neighborhood of $z$, then it has an $n$-complement for some $n\in\mathcal{N}$. We also show the boundedness of complements for $\mathbb{R}$-complementary surface pairs.
\end{abstract}

\date{\today}
\thanks{%2010 \emph{Mathematics Subject Classification}: 14E30, 14J45, 14J17.
}

\maketitle
\pagestyle{myheadings}\markboth{\hfill G. Chen, J. Han, and Q. Xue \hfill}{\hfill Boundedness of complements for log Calabi-Yau threefolds\hfill}

\tableofcontents

\section{Introduction}
We work over the field of complex numbers $\Cc$.

The theory of complements (for Fano varieties) was introduced by Shokurov when he proved the existence of flips for threefolds \cite{Sho92}. It originates from his earlier work on anti-canonical systems on Fano threefolds \cite{Sho79}. The boundedness of complements \cite{Bir19,HLS19,Sho20} played an important role in various contexts in the study of Fano varieties, including the solution of the Borisov-Alexeev-Borisov conjecture (boundedness of Fano varieties) \cite{Bir19,Bir21} and the Yau-Tian-Donaldson conjecture (the existence of K{\"a}hler-Einstein metrics on log Fano pairs) \cite{Xu20,BLX19,LXZ21}. We refer the reader to \cite{Liu18,FMX19,HLS19,Chen20,CGN21,CH21,CZ21,CX22a,HLL22} and references therein for more recent progress and applications.

According to the minimal model program, varieties of general type, Fano varieties and Calabi-Yau varieties form three fundamental classes in birational geometry and are building blocks of algebraic varieties. In this paper, we study the theory of complements for Calabi-Yau type varieties with the coefficient set $[0,1]$ in dimensions 2 and 3. Note that Calabi-Yau type varieties form a large class of varieties which includes both Fano varieties and Calabi-Yau varieties. For Calabi-Yau varieties, since the boundedness of complements implies the boundedness of the non-vanishing index of $K_X$, we expect that the theory of complements will play an important role in the study of Calabi-Yau varieties, including the boundedness of Calabi-Yau varieties. We also remark that replacing a coefficient set which satisfies the descending chain condition (DCC) with the set $[0,1]$ is considered as a very hard problem in the theory of complements.

\medskip

Our first main result is the boundedness of complements for threefold pairs.

\begin{thm}\label{thm:3foldsbddcmpt}
Let $l$ be a positive integer. Then there exists a finite set of positive integers $\cN$ depending only on $l$ satisfying the following. 

Assume that $(X/Z\ni z,B)$ is a threefold pair which has an $\Rr$-complement that is klt over a neighborhood of $z$. Then $(X/Z\ni z,B)$ has an $n$-complement for some $n\in\cN$ such that $l\mid n$. 
\end{thm}

Theorem \ref{thm:3foldsbddcmpt} fails if we remove the assumption ``klt over a neighborhood of $z$''; see \cite[Example 11]{Sho20}. However, if we require the coefficients of the boundaries to lie in a set $\Gamma\subseteq[0,1]$ such that $\Gamma\cap\Qq$ is DCC, then we can remove the klt assumption. 
\begin{thm}\label{thm:lctype}
Let $l$ be a positive integer, and $\Gamma\subseteq[0,1]$ a set such that $\Gamma\cap\Qq$ is DCC. Then there exists a finite set of positive integers $\cN$ depending only on $l$ and $\Gamma$ satisfying the following. 

Assume that $(X/Z\ni z,B)$ is an $\Rr$-complementary threefold pair such that $X$ is of Calabi-Yau type over a neighborhood of $z$ and $B\in\Gamma$. Then $(X/Z\ni z,B)$ has an $n$-complement for some $n\in\cN$ such that $l\mid n$. 
\end{thm}

Here, we say that $X$ is \emph{of Calabi-Yau type} over a neighborhood of $z$, if there exists a boundary $C$ on $X$ such that $(X,C)$ is klt and $K_X+C\sim_{\Rr,Z}0$ over a neighborhood of $z$; see Definition \ref{defn:cytype}.

\medskip

Our last main result is the boundedness of complements for surface pairs where we do not require the pair has a klt $\Rr$-complement nor $\Ii\cap \Qq$ is DCC. Theorem \ref{thm:surfcmpt} completely answers a question of Shokurov \cite[1.3 Conjecture on complements]{Sho00} for surfaces.

\begin{thm}\label{thm:surfcmpt}    
Let $l$ be a positive integer. Then there exists a finite set of positive integers $\cN$ depending only on $l$ satisfying the following.

Assume that $(X/Z\ni z,B)$ is an $\Rr$-complementary surface pair. Then $(X/Z\ni z,B)$ has an $n$-complement for some $n\in\cN$ such that $l\mid n$.
\end{thm}

\medskip
\noindent{\bf Sketch of proofs.}
We now sketch the proofs of Theorems \ref{thm:3foldsbddcmpt} and \ref{thm:surfcmpt}. For convenience, in what follows, we will assume that $l=1$ and $(X,B)$ is a $\Qq$-factorial klt log Calabi-Yau pair, that is, $Z=z=\{pt\}$, $(X,B)$ is $\Qq$-factorial klt and $K_X+B\sim_\Rr0$.

\smallskip

We first sketch the proof of Theorem \ref{thm:surfcmpt}. If $X$ is of Fano type, then $(X,B)$ is $\cN_1$-complementary for some finite set of positive integers $\cN_1$ by Theorem \ref{thm:ftbdd}; here $(X,B)$ being $\cN_1$-complementary means that $(X,B)$ is $n$-complementary for some $n\in\cN_1$ (see Definition \ref{defn complement}). Thus we may assume that $X$ is not of Fano type and $\kappa(X,B-B_{\Phi_1})\le1$, where $\Phi_1:=\Gamma(\cN_1,\{0,1\})$ is a hyperstandard set and $B_{\Phi_1}$ is a $\mathbb{Q}$-divisor with coefficients in $\Phi_1$ such that $0 \le B_{\Phi_1} \le B$ (see Definition \ref{defn:phiset}). Suppose that $\kappa(X,B-B_{\Phi_1})=\kappa(X,B-B_{\Phi_2})=1$, where $\cN_2$ is a finite set of positive integers given by Theorem \ref{thm:curcmpt} and  $\Phi_2:=\Gamma(\cN_1\cup\cN_2,\{0,1\})$. In this case, we claim that $(X,B)$ is $\cN_2$-complementary.
Indeed, although $X$ is not of Fano type, by Lemma \ref{lem:runantiMMP} we can still run an MMP on $-(K_X+B_{\Phi_2})$ and get a good minimal model $X'$ such that $-(K_{X'}+B_{\Phi_2}')$ is semi-ample and hence defines a contraction $\pi':X'\to Z'$, where $D'$ denotes the strict transform of $D$ on $X'$ for any $\Rr$-divisor $D$ on $X$. Then we run an MMP on $-(K_{X'}+B_{\Phi_1}')$ over $Z'$ and reach a model $X''$ on which $-(K_{X''}+B_{\Phi_1}'')$ is semi-ample over $Z'$, where $D''$ denotes the strict transform of $D$ on $X''$ for any $\Rr$-divisor $D$ on $X$.
As $\kappa(X,B-B_{\Phi_1})=\kappa(X,B-B_{\Phi_2})=1$, the natural morphism $\pi'':X''\to Z'$ is the contraction defined by $-(K_{X''}+B_{\Phi_1}'')$ over $Z'$. By the similar arguments as in \cite[Proposition 6.3]{Bir19} and using Effective Adjunction \cite[Conjecture 7.13.3 and Theorem 8.1]{PS09}, there exists a positive integer $p$ which only depends on $\Phi_1$ such that $p(K_{X''}+B_{\Phi_1}'')\sim p(\pi'')^*(K_{Z'}+B_{Z'}^{(1)}+\bM_{\pi'',Z'})$ and $p\bM_{\pi''}$ is base point free, where $B_{Z'}^{(1)}$ and $\bM_{\pi''}$ are given by the canonical bundle formula for $(X'',B_{\Phi_1}'')$ over $Z'$ in Proposition \ref{prop:cbfindex}. It follows that $p(K_{X'}+B_{\Phi_2}')\sim p(\pi')^*(K_{Z'}+B_{Z'}^{(2)}+\bM_{\pi',Z'})$ and $p\bM_{\pi'}$ is base point free, where $B_{Z'}^{(2)}$ and $\bM_{\pi'}$ are given by the canonical bundle formula for $(X',B_{\Phi_2}')$ over $Z'$. We may assume that $p\mid n$ for any $n \in \cN_2$. As $p\bM_{\pi'}$ is base point free, one can find an effective $\Qq$-divisor $M_{Z'}$ such that $pM_{Z'}\sim p\bM_{\pi',Z'}$, $M_{Z'}$ has no common components with $B_{Z'}^{(2)}$, and $(Z',B_{Z'}^{(2)}+M_{Z'})$ has an $n$-complement for some $n\in\cN_2$. Then we can lift this complement to $X$ and get an $n$-complement of $(X,B)$; see Proposition \ref{prop:cbflift1}. If $\kappa(X,B-B_{\Phi_2})=0$, then we can easily show that $n_0(K_X+B)\sim0$ for some positive integer $n_0$ which only depends on $\Phi_2$; see Lemma \ref{lem:global}. Hence $\cN_1\cup\cN_2\cup\{n_0\}$ has the required property.

\smallskip

Now we sketch the proof of Theorem \ref{thm:3foldsbddcmpt}. The main strategy is similar. One of the key steps is to construct a positive integer $n_0$ and finite sets of positive integers $\cN_{i}$ ($i=1,2,3$) such that 
\begin{enumerate}
    \item if $\kappa(X,B-B_{\Phi_1})=3$, then $(X,B)$ is $\cN_1$-complementary,
    \item if $\kappa(X,B-B_{\Phi_1})=\kappa(X,B-B_{\Phi_2})=2$, then $(X,B)$ is $\cN_2$-complementary,
    \item if $\kappa(X,B-B_{\Phi_2})=\kappa(X,B-B_{\Phi_3})=1$, then $(X,B)$ is $\cN_3$-complementary, and
    \item if $\kappa(X,B-B_{\Phi_3})=0$, then $(X,B)$ is $n_0$-complementary,
\end{enumerate}
where $\Phi_i:=\Gamma(\cup_{j=1}^i\cN_j,\{0,1\})$ for any $1\le i\le 3$; see Section \ref{sec6} for the details. However, there are some issues when we construct these finite sets. One issue is that when we apply the canonical bundle formula, Effective Adjunction is still open when the relative dimension is $\ge2$. But in our setting we can give a positive answer to Effective Adjunction; see Proposition \ref{prop:bsa} for the details. On the other hand, there is also an issue when we try to lift complements from lower dimensional varieties. More precisely, it may happen that some components of $\Supp B$ have images of codimension $\ge2$ in $Z$. Therefore we must lift complements more carefully; see Proposition \ref{prop:cbflift1} and Section \ref{sec6} for the details.

\medskip

\noindent\textit{Structure of the paper.} 
We outline the organization of the paper. In Section \ref{sec2}, we introduce some notation and tools which will be used in this paper, and prove certain basic results. In Section \ref{sec3}, we recall the canonical bundle formula, some well-known results, as well as some new results. In Section \ref{sec4}, we prove the boundedness of complements for sdlt curves. In Section \ref{sec5}, we prove Theorem \ref{thm:surfcmpt}. In Section \ref{sec6}, we prove Theorem \ref{thm:3foldsbddcmpt}. In Section \ref{sec7}, we prove Theorem \ref{thm:lctype}.

\medskip

\noindent{\bf Acknowledgements.} 
The authors would like to thank Qianyu Chen, Chen Jiang, Jihao Liu, Wenfei Liu, Lingyao Xie and Chenyang Xu for valuable discussions and suggestions. The third named author would also like to thank his advisor Christopher D. Hacon for his constant support and useful suggestions. The first named author was supported by the China post-doctoral grants BX2021269 and 2021M702925. The second named author was supported by National Key Research and Development Program of China (Grant No. 2020YFA0713200). The third named author was partially supported by NSF research grants no: DMS-1801851, DMS-1952522 and by a grant from the Simons Foundation; Award Number: 256202. Finally the authors would like to thank the referees for their careful reading of this paper and the many useful comments.

\section{Preliminaries}\label{sec2}
\subsection{Arithmetic of sets}

\begin{defn}\label{defn:phiset}
(1) We say that a set $\Gamma\subseteq[0,1]$ satisfies the \emph{descending chain condition} (DCC) if any decreasing sequence $a_1\ge a_2\ge\cdots$ in $\Gamma$ stabilizes. We say that $\Gamma$ satisfies the \emph{ascending chain condition} (ACC) if any increasing sequence $a_1\le a_2\le \cdots$ in $\Gamma$ stabilizes.

\smallskip

(2) Suppose that $\fR\subseteq[0,1]\cap\Qq$ is a finite set. We define
$$\Phi(\fR):=\left\{1-\frac{r}{l}\mid r\in\fR,l\in\Zz_{>0}\right\}\cap[0,1]$$
to be the set of \emph{hyperstandard multiplicities} associated to $\fR$ (cf. \cite[2.2]{Bir19}). We may say that $\Phi(\fR)$ is the \emph{hyperstandard set} associated to $\fR$. When we say $\Phi\subseteq[0,1]\cap\Qq$ is a hyperstandard set, we mean that $\Phi=\Phi(\fR)$ for some finite set $\fR\subseteq[0,1]\cap\Qq$. We usually assume $0,1\in\fR$ without mentioning, so $\Phi(\{0,1\})\subseteq\Phi(\fR)$.

\smallskip

(3) (cf. \cite[Page 30]{Sho20}) Let $\cN\subseteq\Zz_{>0}$, $\fR\subseteq[0,1]\cap\Qq$ be two finite sets, and $\Phi:=\Phi(\fR)$. We define 
$$\Gamma(\cN,\Phi(\fR)):=\bigg\{1-\frac{r}{l}+\frac{1}{l}\sum_{n\in\cN}\frac{m_n}{n+1}\mid r\in\fR,l\in\Zz_{>0},m_n\in\Zz_{\ge0}\bigg\}\cap[0,1].$$
By Remark \ref{rem:phiset}(1), $\Gamma(\cN,\Phi(\fR))$ is independent of the choice of $\fR$. Hence we may write $\Gamma(\cN,\Phi)$ instead of $\Gamma(\cN,\Phi(\fR))$ for convenience.
By Remark \ref{rem:phiset}(2), $\Gamma(\cN,\Phi)$ is a hyperstandard set. In particular, it is a DCC set whose only accumulation point is $1$. Then for any $b\in[0,1]$, we define
$$b_{\cN\_\Phi}:=\max\left\{b'\mid b'\le b,\, b'\in\Ii(\cN,\Phi)\right\}.$$ 

If $\cN=\{n\}$ (respectively $\cN=\emptyset$), we may write $b_{n\_\Phi}$ (respectively $b_{\Phi}$) rather than $b_{\cN\_\Phi}$. 
\end{defn}
\begin{rem}\label{rem:phiset}
(1) If $\fR'\subseteq[0,1]\cap\Qq$ is a finite set such that $\Phi(\fR)=\Phi(\fR')$, then $\Gamma(\cN,\Phi(\fR))=\Gamma(\cN,\Phi(\fR')).$ Indeed, for any $r'\in\fR'$, there exist $r\in\fR$ and $l\in\Zz_{>0}$ such that $r'=r/l$. Thus $\Gamma(\cN,\Phi(\fR))\supseteq\Gamma(\cN,\Phi(\fR'))$, and the converse inclusion follows similarly.

\smallskip

(2) $\Ii(\cN,\Phi)$ is the hyperstandard set associated to the following finite set
$$\fR'' := \left\{r-\sum_{n\in\cN}\frac{m_n}{n+1}\mid r\in\fR,m_n\in\Zz_{\ge0}\right\}\cap[0,1].$$
Indeed,
\begin{align*}
\Phi(\fR'')&=\left\{1-\frac{r''}{l}\mid r''\in\fR'',l\in\Zz_{>0}\right\}\cap[0,1]\\
&= \left\{1-\frac{r}{l}+\frac{1}{l}\sum_{n\in\cN}\frac{m_n}{n+1}\mid r\in\fR,l\in\Zz_{>0},m_n\in\Zz_{\ge0}\right\}\cap[0,1] \\&= \Ii(\cN,\Phi).
\end{align*}

(3) If $\cN_1$ and $\cN_2$ are two finite sets of positive integers, then
$$\Gamma(\cN_{1}\cup\cN_2,\Phi)=\Gamma(\cN_2,\Gamma(\cN_1,\Phi)).$$ 
Indeed, let
$$\fR_1 := \left\{r-\sum_{n\in\cN_1}\frac{m_n}{n+1}\mid r\in\fR,m_n\in\Zz_{\ge0}\right\}\cap[0,1].$$
Then $\Gamma(\cN_1,\Phi)=\Phi(\fR_1)$ by (2).
Therefore
\begin{align*}
&\Gamma(\cN_2,\Gamma(\cN_1,\Phi))=\Gamma(\cN_2,\Phi(\fR_1))\\
=&\bigg\{1-\frac{r_1}{l}+\frac{1}{l}\sum_{n\in\cN_2}\frac{m_n}{n+1}\mid r_1\in\fR_1,l\in\Zz_{>0},m_n\in\Zz_{\ge0}\bigg\}\cap[0,1]\\
=&\bigg\{1-\frac{1}{l}(r-\sum_{n'\in\cN_1}\frac{m'_{n'}}{n'+1})+\frac{1}{l}\sum_{n\in\cN_2}\frac{m_n}{n+1}\mid r\in\fR,l\in\Zz_{>0},m'_{n'},m_n\in\Zz_{\ge0}\bigg\}\cap[0,1]\\
=&\bigg\{1-\frac{r}{l}+\frac{1}{l}(\sum_{n'\in\cN_1}\frac{m'_{n'}}{n'+1}+\sum_{n\in\cN_2}\frac{m_n}{n+1})\mid r\in\fR,l\in\Zz_{>0},m_n,m'_{n'}\in\Zz_{\ge0}\bigg\}\cap[0,1]\\
=&\Gamma(\cN_{1}\cup\cN_2,\Phi).
\end{align*}
\end{rem}

\smallskip

The following lemma was observed by the second named author. It will play an important role in the proof of the main theorems.

\begin{lem}\label{lem:ncomplneq}
Assume that $\fR\subseteq[0,1]\cap\Qq$ is a finite set, $\Phi:=\Phi(\fR)$, and $n$ is a positive integer such that $n\fR\subseteq\Zz$. Then for any $b\in[0,1]$, we have
$$\frac{\lf(n+1)\{b\}\rf}{n}+\lf b\rf \ge b_{n{\_}\Phi}.$$
\end{lem}
\begin{proof}
Without loss of generality, we may assume that $1>b=b_{n{\_}\Phi}=1-\frac{r}{l}+\frac{m}{l(n+1)}\in\Ii(\{n\},\Phi)$ for some $r\in\fR,l\in\Zz_{>0}$ and $m\in\Zz_{\ge0}$. It suffices to show that
$$\lf1-\frac{(n+1)r}{l}+\frac{m}{l}\rf\ge-\frac{nr}{l}+\frac{mn}{l(n+1)}.$$
Suppose on the contrary that there exists an integer $k$ such that
$$1-\frac{(n+1)r}{l}+\frac{m}{l}<k\text{ and } -\frac{nr}{l}+\frac{mn}{l(n+1)}>k-1.$$
The first inequality above gives us that $l-lk+m<(n+1)r$, and thus $l-lk+m\le nr$ as $nr\in\Zz$. Therefore, we have
$$\frac{mn}{n+1}> (k-1)l+nr\ge(k-1)l+l-lk+m=m,$$
a contradiction.
\end{proof}

\begin{lem}\label{lem:ncomplneq2}
Let $\cN$ be a finite set of positive integers, $\Phi$ a hyperstandard set, and $n\in\cN$. Suppose that $b,b^+\in[0,1]$ such that $nb^+\in\Zz$ and 
$$nb^+\ge \lf(n+1)\{b_{\cN\_\Phi}\}\rf+n\lf b_{\cN\_\Phi}\rf.$$ 
Then $nb^+\ge \lf(n+1)\{b\}\rf+n\lf b\rf$.
\end{lem}
\begin{proof}
If $b=1$, then $b^+=b_{\cN\_\Phi}=1$ and there is nothing to prove. 

If $b<1$, then $b_{\cN\_\Phi}\le b<1$.
It suffices to show that 
$
\rddown{(n+1)b}=
\rddown{(n+1)b_{\cN\_\Phi}}$. Let $$b':=\max\{\frac{l}{n+1}\mid \frac {l}{n+1}\le b, l\in\Zz_{\ge0}\}\in \Gamma(\cN,\Phi).$$ 
By the construction, $\rddown{(n+1)b'}=\rddown{(n+1)b}$, which implies that $\rddown{(n+1)b}=\rddown{(n+1)b_{\cN\_\Phi}}$ as $b'\le b_{\cN\_\Phi}\le b$. 
\end{proof}

\subsection{Divisors}

Let $\F$ be either the rational number field $\Qq$ or the real number field $\Rr$. Let $X$ be a normal variety and $\WDiv(X)$ the free abelian group of Weil divisors on $X$. Then an \emph{$\F$-divisor} is defined to be an element of $\WDiv(X)_\F:=\WDiv(X)\otimes\F$.

A \emph{b-divisor} on $X$ is an element of the projective limit
$$\textbf{WDiv}(X)=\lim_{Y\to X}\WDiv (Y),$$
where the limit is taken over all the pushforward homomorphisms $f_*:\WDiv (Y)\to\WDiv(X)$ induced by proper birational morphisms $f:Y\to X$. In other words, a b-divisor $D$ on $X$ is a collection of Weil divisors $\textbf{D}_Y$ on higher models $Y$ of $X$ that are compatible with respect to pushforward. The divisor $\textbf{D}_Y$ is called the \emph{trace} of $\textbf{D}$ on the birational model $Y$. A \emph{b-$\F$-divisor} is defined to be an element of $\textbf{WDiv}(X)\otimes\F$, and the trace of a b-$\F$-divisor is defined similarly.

The \emph{Cartier closure} of an $\F$-Cartier $\F$-divisor $D$ on $X$ is the b-$\F$-divisor $\overline{D}$ with trace $\overline{D}_Y=f^*D$ for any proper birational morphism $f:Y\to X$. A b-$\F$-divisor $\textbf{D}$ on $X$ is \emph{b-$\F$-Cartier} if $\textbf{D}=\overline{D_{Y}}$ where $\overline{D_Y}$ is an $\F$-Cartier $\F$-divisor on a birational model $Y$ of $X$; in this situation, we say $\textbf{D}$ \emph{descends to} $Y$. Moreover, if $D_Y$ is a Cartier divisor, then we say ${\bf D}$ is \emph{b-Cartier}. Let $X\to Z$ be a projective morphism. Then a b-$\F$-divisor is \emph{nef} (respectively \emph{base point free}) \emph{over $Z$} if it descends to a nef (respectively base point free) over $Z$ $\F$-divisor on some birational model of $X$.

\smallskip

Assume that $\Gamma\subseteq[0,1]$ is a set, and $B:=\sum b_iB_i$, $B':=\sum b_i'B_i$ are two $\Rr$-divisors on $X$, where $B_i$ are prime divisors. By $B\in\Gamma$ we mean $b_i\in\Gamma$ for any $i$. We define
$\lf B\rf:=\sum\lf b_i\rf B_i,\, \{B\}:=\sum\{b_i\}B_i,\, ||B||:=\max\{|b_i|\}$,
and $B\wedge B':=\sum\min\{b_i,b_i'\}B_i$. Assume that $\cN$ is a finite set of positive integers and $\Phi$ is a hyperstandard set. We define
$$B_{\cN\_\Phi}:=\sum (b_i)_{\cN\_\Phi}B_i.$$
If $\cN=\{n\}$ (respectively $\cN=\emptyset$), we may write $B_{n\_\Phi}$ (respectively $B_{\Phi}$) instead of $B_{\cN\_\Phi}$. 

\begin{defn}
(1) We say $\pi: X \to Z$ is a \emph{contraction} if $X$ and $Z$ are normal quasi-projective varieties, $\pi$ is a projective morphism, and $\pi_*\Oo_X = \Oo_Z$.

(2) We say that a birational map $\phi: X \dashrightarrow Y$ is a \emph{birational contraction} if $\phi$ is projective and $\phi^{-1}$  does not contract any divisors.
\end{defn}

\begin{lem}\label{lem:relamp}
Suppose that $\tau:Z''\to Z'$ and $Z'\to Z$ are contractions. Suppose that $H''$ (respectively $H'$) is an $\Rr$-Cartier $\Rr$-divisor on $Z''$ (respectively $Z'$) which is ample over $Z'$ (respectively $Z$). Then $\epsilon H''+\tau^*H'$ is ample over $Z$ for any $0<\epsilon\ll1$.
\end{lem}
\begin{proof}
Pick any closed point $z\in Z$. Let $Z'_z:=Z'\times_Z\{z\}$, $Z''_z:=Z''\times_Z\{z\}$, $H'_z:=H'|_{Z'_z}$, and $H''_z:=H''|_{Z''_z}$. By assumption $H''_z$ is ample over $Z'_z$ and $H'_z$ is ample. According to \cite[Proposition 1.45]{KM98}, $\epsilon_zH''_z+(\tau|_{Z''_z})^*H'_z$ is ample for any $0<\epsilon_z\ll1$. In particular, $\epsilon_z H''+\tau^*H'$ is ample over $z$. It follows that $\epsilon_z H''+\tau^*H'$ is ample over some neighborhood of $z$ by \cite[Theorem 1.2.17]{Lazpos1}. Since $Z$ is quasi-compact, the lemma follows.
\end{proof}

%%%%%%%%%%%%%%%%%%%%%%%%%%%%%

\subsection{Generalized pairs and singularities}
In this paper, we usually discuss the (sub-)pair in the relative setting $(X/Z\ni z,B)$; we refer the reader to \cite[\S2]{CH21} (cf. \cite{KM98,BCHM10}). Moreover, if the (sub-)pair $(X/Z\ni z,B)$ is (sub-)lc over $z$ for any $z\in Z$, then we say $(X/Z,B)$ is (sub-)lc.

Here we briefly discuss the analogous concepts for generalized pairs, and refer the reader to \cite{BZ16,HL21,HL18} for further details.

\begin{defn}
A \emph{generalized pair} (\emph{g-pair} for short) $(X/Z,B+\bM)$ consists of a contraction $X\to Z$, an effective $\Rr$-divisor $B$ on $X$, and a nef$/Z$ b-$\Rr$-divisor $\bM$ on $X$, such that $K_X+B+\bM_X$ is $\Rr$-Cartier.

\smallskip

Let $(X/Z,B+\bM) $ be a g-pair and $f:W \to X$ a log resolution of $(X,\Supp B)$ to which $\bM$ descends. We may write 
$$K_{W}+B_W+\bM_W=f^*(K_X+B+\bM_X)$$
for some $\Rr$-divisor $ B_W$ on $W$. Let $E$ be a prime divisor on $W$. The \emph{log discrepancy} of $E$ with respect to $(X,B+\bM)$ is defined as
$$a(E,X,B+\bM):=1-\mult_EB_W.$$
We say $(X/Z,B+\bM)$ is \emph{generalized lc} or \emph{glc} (respectively \emph{generalized klt} or \emph{gklt}) if $a(E,X,B+\bM)\ge0$ (respectively $>0$) for any prime divisor $E$ over $X$.

We say that two g-pairs $(X/Z,B+\bM)$ and $(X'/Z,B'+\bM')$ are \emph{crepant} if $X$ is birational to $X'$, $\bM=\bM'$, and $f^*(K_X+B+\bM_X)=(f')^*(K_{X'}+B'+\bM'_{X'})$ for some common resolution $f:W\to X$ and $f':W\to X'$. We also call $(X'/Z,B'+\bM)$ a \emph{crepant model} of $(X/Z,B+\bM)$.
\end{defn}

\begin{lem}\label{lem:dcckltsing}
Let $d$ be a positive integer and $\Gamma\subseteq[0,1]$ a DCC set. Then there is a positive real number $\epsilon$ depending only on $d$ and $\Gamma$ satisfying the following. Assume that $(X,B)$ is a projective klt pair of dimension $d$ such that $K_X+B\sim_{\Rr}0$ and $B\in\Gamma$. Then $(X,B)$ is $\epsilon$-lc.
\end{lem}
\begin{proof}
The lemma follows from \cite[Lemma 2.48]{Bir19}.
\end{proof}

\begin{defn}
Let $X\to Z$ be a contraction and $D$ an $\Rr$-Cartier $\Rr$-divisor on $X$. We denote by $\kappa(X,D)$ and $\kappa(X/Z,D)$ the \emph{Iitaka dimension} and \emph{relative Iitaka dimension} of $D$ respectively; see \cite[II, \S3.b and \S3.c]{Nak04}.
\end{defn}

\begin{defn}
Let $X\to Z$ be a contraction, $D$ an $\Rr$-Cartier $\Rr$-divisor on $X$, and $\phi:X\dashrightarrow Y$ a birational contraction of normal quasi-projective varieties over $Z$. We say that $Y$ is a \emph{good minimal model} of $D$ over $Z$, if $\phi$ is $D$-negative, $D_Y$ is semiample over $Z$, and $Y$ is $\Qq$-factorial, where $D_Y$ is the strict transform of $D$ on $Y$.
\end{defn}

\begin{lem}\label{lem:finitemmpiitdim}
Let $X\to Z$ be a contraction, and $D_2\ge D_1$ two effective $\Rr$-Cartier $\Rr$-divisors on $X$. Suppose that $X\dashto X'$ is a sequence of steps of the $D_1$-MMP over $Z$. Let $D_2'$ be the strict transform of $D_2$ on $X'$. Then $\kappa(X/Z,D_2)=\kappa(X'/Z,D_2')$.
\end{lem}
\begin{proof}
Pick a positive real number $\epsilon$ such that $X\dashto X'$ is also a sequence of steps of the $(D_1+\epsilon D_2)$-MMP over $Z$. As $\Supp (D_1+\epsilon D_2)=\Supp D_2$, one can see that
$$\kappa(X'/Z,D_2')=\kappa(X'/Z,D_1'+\epsilon D_2')=\kappa(X/Z,D_1+\epsilon D_2)=\kappa(X/Z,D_2),$$
where $D_1'$ is the strict transform of $D_1$ on $X'$.
\end{proof}

%%%%%%%%%%%%%%%%%%%%%%%%%%%%%%%%%%%%%%%%
%%%%%%%%%%%%%%%%%%%%%%%%%%%%%%%%%%%%%%%%
\subsection{Complements}

\begin{defn}\label{defn complement}
We say that a pair $(X/Z\ni z,B^+)$ is an \emph{$\Rr$-complement} of $(X/Z\ni z,B)$ if $(X,B^+)$ is lc, $B^+\ge B$, and $K_X+B^+\sim_\Rr0$ over a neighborhood of $z$. In addition if $(X,B^+)$ is klt over a neighborhood of $z$, then we call $(X/Z\ni z,B^+)$ a \emph{klt $\Rr$-complement} of $(X/Z\ni z,B)$. 
    
Let $n$ be a positive integer. An \emph{$n$-complement} of $(X/Z\ni z,B)$ is a pair $(X/Z\ni z,B^+),$ such that over some neighborhood of $z,$ we have 
\begin{enumerate} 
  \item $(X,B^+)$ is lc,          
  \item $n(K_X+B^+)\sim 0,$ and          
  \item $nB^+\ge n\lfloor B\rfloor+\lfloor(n+1)\{B\}\rfloor.$
\end{enumerate} 
We say that $(X/Z\ni z, B^{+})$ is a \emph{monotonic $n$-complement} of $(X/Z\ni z, B)$ if we additionally have $B^{+}\ge B$.

Let $\cN$ be a non-empty set of positive integers. We say that $(X/Z\ni z,B)$ is \emph{$\cN$-complementary} (respectively \emph{$n$-complementary}, \emph{$\Rr$-complementary}) if $(X/Z\ni z,B)$ has an $m$-complement (respectively $n$-complement, $\Rr$-complement) for some $m\in\cN$. If $(X/Z\ni z,B)$ has an $\Rr$-complement (respectively $n$-complement) for any $z\in Z$, then we say that $(X/Z,B)$ is \emph{$\Rr$-complementary} (respectively \emph{$n$-complementary}).
\end{defn}

Note that if $z' \in \bar{z}$ is a closed point and $(X/Z\ni z',B^+)$ is an $\Rr$-complement (respectively an $n$-complement) of $(X/Z\ni z',B)$, then $(X/Z\ni z,B^+)$ is an $\Rr$-complement (respectively an $n$-complement) of $(X/Z\ni z,B)$. Hence when proving the existence of complements we may assume that $z \in Z$ is a closed point.

\smallskip

The following lemma is well-known to experts. We will use the lemma frequently without citing it in this paper.
\begin{lem}[cf. {\cite[6.1]{Bir19}, \cite[Lemma 2.11]{CH21}}]\label{lemma pullbaclcomplement}
Let $n$ be a positive integer. Assume that $(X/Z\ni z,B)$ is a pair, $\psi:X\dashrightarrow X'$ is a birational contraction over $Z$, and $B'$ is the strict transform of $B$ on $X'$.
\begin{enumerate}
  \item If $(X/Z\ni z,B)$ is $\Rr$-complementary (respectively $n$-complementary), then so is $(X'/Z\ni z,B')$.
  \item Suppose $\psi$ is $-(K_X+B)$-non-positive. If $(X'/Z\ni z,B')$ has an $\Rr$-complement (respectively a monotonic $n$-complement), then so does $(X/Z\ni z,B)$.
\end{enumerate}
\end{lem}

The following lemma is an easy consequence of Lemmas \ref{lem:ncomplneq} and \ref{lem:ncomplneq2}. 

\begin{lem}\label{lem:N_Phicompl}
Let $\cN\subseteq\Zz_{>0},$ $\fR\subseteq[0,1]\cap\Qq$ be two finite sets, $\Phi:=\Phi(\fR)$, and $n$ a positive integer such that $n\fR\subseteq\Zz_{\ge0}$. Assume that $(X/Z\ni z,B)$ is a pair.
\begin{enumerate}
    \item If $(X/Z\ni z,B_{\cN\_\Phi})$ is $\cN$-complementary, then so is $(X/Z\ni z,B)$.
    \item Any $n$-complement of $(X/Z\ni z,B)$ is a monotonic $n$-complement of $(X/Z\ni z,B_{n\_\Phi})$.
\end{enumerate}
\end{lem}

We will use the following lemma frequently in this paper.

\begin{lem}\label{lem:runantiMMP}
Let $\Phi\subseteq[0,1]\cap\Qq$ be a hyperstandard set. Assume that $(X/Z\ni z,B)$ is a pair which is an $\Rr$-complement of itself.
If either
\begin{itemize}
    \item $\dim X=2$ and $X$ is $\Qq$-factorial, or
    \item $\dim X=3$ and $(X,B)$ is dlt over a neighborhood of $z$, 
\end{itemize}
 then $-(K_X+B_\Phi)$ has a good minimal model over a neighborhood of $z$.
\end{lem}
\begin{proof}
According to \cite[Lemma 4.2]{CH21}, possibly shrinking $Z$ near $z$, there exist a positive real number $u$ and a surface (respectively threefold) pair $(X,\Delta)$, such that the coefficients of $\Delta$ are at most 1 (respectively $(X,\Delta)$ is dlt) and $-u(K_X+B_\Phi)\sim_{\Rr,Z} K_X+\Delta$. In both cases, we can run an MMP on $K_X+\Delta$ over $Z$ and reach a good minimal model $X'$ over $Z$ by \cite{KMM85,Fuj12,Tan18}. It is clear that $X'$ is a good minimal model of $-(K_X+B_\Phi)$ as $-u(K_X+B_\Phi)\sim_{\Rr,Z} K_X+\Delta$. This finishes the proof.
\end{proof}

\subsection{Boundedness of complements}
We propose a conjecture on the boundedness of complements and collect some useful results.

For a positive integer $l$ and a non-empty set $\cN\subseteq\Zz_{>0}$, we say $\cN$ is divisible by $l$, denoted by $l\mid \cN$, if $l\mid n$ for any $n\in\cN$.

\begin{conj}\label{conj:bddcpt2}
Let $d,l$ be two positive integers and $\Phi\subseteq[0,1]\cap\Qq$ a hyperstandard set. Then there exists a finite set of positive integers $\cN$ divisible by $l$ depending only on $d,l$ and $\Phi$ satisfying the following.

Assume that $(X/Z\ni z,B)$ is a pair of dimension $d$ such that $(X/Z\ni z,B_{\cN\_\Phi})$ has an $\Rr$-complement which is klt over a neighborhood of $z$. Then $(X/Z\ni z,B)$ is $\cN$-complementary.
\end{conj}

\begin{rem}
\begin{enumerate}
  \item In Conjecture \ref{conj:bddcpt2}, we do not assume $(X/Z\ni z,B)$ is lc. 
  \item One can not remove the klt assumption in Conjecture \ref{conj:bddcpt2} when $d\ge3$; see \cite[Example 11]{Sho20}. However, we will show Conjecture \ref{conj:bddcpt2} for $\Rr$-complementary surface pairs without the klt assumption; see Theorem \ref{thm:surfcmpt}.
\end{enumerate}
\end{rem}

\begin{lem}\label{lem:global}
Let $\Phi\subseteq[0,1]\cap\Qq$ be a hyperstandard set. Then there exists a positive integer $n$ depending only on $\Phi$ satisfying the following.

Assume that $(X,B)$ is a projective $\Qq$-factorial dlt pair of dimension $\le3$ such that $K_X+B\sim_\Rr0$ and $\kappa(X,B-B_\Phi)=0$. Then $n(K_X+B)\sim0.$
\end{lem}

\begin{proof}
By Lemma \ref{lem:runantiMMP}, we may run an MMP on $-(K_X+B_\Phi) \sim_\Rr B - B_\Phi$ which terminates with a good minimal model $X'$. Let $B'$ be the strict transform of $B$ on $X'$. Since $\kappa(X,B-B_\Phi)=0$ and $B'-B'_\Phi$ is semi-ample, we see that $B'=B_{\Phi}'$ and therefore $K_{X'}+B_{\Phi}'\sim_\Rr0$. By \cite[Proposition 6.4, Theorem 1.1]{CH21} and \cite[Theorem 2.12]{CHL22}, there is a positive integer $n$ depending only on $\Phi$ such that
$$n(K_{X'}+B')=n(K_{X'}+B_{\Phi}')\sim0.$$
It follows that $n(K_X+B)\sim0$ as $(X,B)$ and $(X',B')$ are crepant.
\end{proof}

We will use the following results on the boundedness of complements.

\begin{thm}[cf. {\cite[Theorem 16]{Sho20}}]\label{thm:ftbdd}
Let $d,l$ be two positive integers and $\Phi\subseteq[0,1]\cap\Qq$ a hyperstandard set. Then there exists a finite set of positive integers $\cN$ divisible by $l$ depending only on $d,l$ and $\Phi$ satisfying the following.

Assume that $(X/Z\ni z,B)$ is a pair of dimension $d$ such that $X$ is of Fano type over $Z$ and $(X/Z\ni z,B_{\cN\_\Phi})$ has a klt $\Rr$-complement. Then $(X/Z\ni z,B)$ is $\cN$-complementary.
\end{thm}

\begin{thm}[{\cite[Theorem 1.3]{CH21}}]\label{thm:curcmpt}
Let $l$ be a positive integer. Then there exists a finite set of positive integers $\cN$ divisible by $l$ depending on $l$ satisfying the following.
If $(X/Z\ni z,B)$ is an $\Rr$-complementary curve pair, then $(X/Z\ni z,B)$ is $\cN$-complementary.
\end{thm}

%%%%%%%%%%%%%%%%%%%%%%%%%%%%%%%%%%%%%%%%
%%%%%%%%%%%%%%%%%%%%%%%%%%%%%%%%%%%%%%%%

\section{Canonical bundle formulas}\label{sec3}
\subsection{Canonical bundle formulas}
For the definition and basic properties of the canonical bundle formula, we refer the reader to \cite{Bir19,Fil20,HanLiu20,JLX22}. Briefly speaking, suppose that $(X/Z,B)$ is a sub-pair and $\phi:X\to T$ is a contraction over $Z$, such that $(X,B)$ is lc over the generic point of $T$ and $K_X+B\sim_{\Rr,T}0$. Then there exist a uniquely determined $\Rr$-divisor $B_T$ and a nef over $Z$ b-$\Rr$-divisor $\bM_\phi$ which is determined only up to $\Rr$-linear equivalence, such that $(T/Z,B_T+\bM_\phi)$ is a g-sub-pair and
$$K_X+B\sim_\Rr\phi^*(K_T+B_T+\bM_{\phi,T}).$$
Here $B$ (respectively $\bM_\phi$) is called the \emph{discriminant part} (respectively a \emph{moduli part}) of the canonical bundle formula for $(X/Z,B)$ over $T$. Moreover, if $(X/Z,B)$ is an lc (respectively klt) pair, then $(T/Z,B_T+\bM_\phi)$ is a glc (respectively gklt) g-pair.

It is worthwhile to point out that $\bM_\phi$ only depends on $(X,B)$ over the generic point of $T$ (cf. \cite[3.4(2)]{Bir19}), and there are many choices of $\bM_\phi$, some of which could behave badly. But we can always choose one with the required properties, e.g. Propositions \ref{prop:cbfindex} and \ref{prop:bsa}.
\begin{lem}\label{lem:cbfcptoverbase}
Notation as above.
\begin{enumerate}
  \item Assume that $(X,B)$ is a klt pair. Then there exists a crepant model $(\tilde{T},B_{\tilde{T}}+\bM_{\phi}) \to (T,B_T+\bM_\phi)$ such that for any prime divisor $P\subseteq\Supp B$ which is vertical over $T$, the image of $P$ on $\tilde{T}$ is a prime divisor.
  \item Suppose that there is an $\Rr$-divisor $G_T$ on $T$ such that $(T,B_T+G_T+\bM_\phi)$ is a sub-glc g-sub-pair. If we let $G:=\phi^*G_T$, then $(X,B+G)$ is sub-lc.
\end{enumerate}
\end{lem}
\begin{proof}
(1) According to \cite[Theorem B.6]{Hu20} (cf. \cite[Theorem 0.3]{AK00}, \cite[Theorem 2]{Kaw15}, and \cite[Theorem 2.8]{HL21}), there exist birational morphisms $X'\to X$ and $T'\to T$ such that $X'\to T'$ is an equidimensional contraction. In particular, for any prime divisor $P\subseteq\Supp B$ which is vertical over $T$, the image $Q$ of $P$ on $T'$ is a prime divisor. Moreover, by the canonical bundle formula, $a(Q,T,B_T+\bM_\phi)<1$ as $a(P,X,B)<1$. Since $(X,B)$ is a klt pair, $(T,B_T+\bM_\phi)$ is a gklt g-pair. So (1) holds by \cite[Lemma 4.6]{BZ16}.

(2) Suppose on the contrary that $(X,B+G)$ is not sub-lc. Let $P''$ be a non-sub-lc place of $(X,B+G)$, i.e., $a(P'',X,B+G)<0$. It is clear that $\Center_X(P'')\subseteq \Supp G$ which is vertical over $T$. We can find birational morphisms $f:X''\to X$ and $g:T''\to T$ such that $X''\to T''$ is a contraction, $P''$ is a prime divisor on $X''$, and the image $Q''$ of $P''$ on $T''$ is a prime divisor (cf. \cite[VI, Theorem 1.3]{Kol96}). We may write $K_{X''}+\Delta'':=f^*(K_X+B+G)\text{ and }K_{T''}+\Delta_{T''}+\bM_{\phi,T''}:=g^*(K_T+B_T+G_T+\bM_{\phi,T})$ for some $\Rr$-divisors $\Delta''$ and $\Delta_{T''}$. Then $\Delta_{T''}$ is the discriminant part of the canonical bundle formula for $(X''/Z,\Delta'')$ over $T''$; see \cite[Lemma 7.4(ii)]{PS09}. Since $(T,B_T+G_T+\bM_\phi)$ is sub-glc, $\mult_{Q''}\Delta_{T''}\le1$. By the definition of the canonical bundle formula, $(X'',\Delta'')$ is sub-lc over the generic point of $Q''$. In particular, $\mult_{P''}\Delta''=1-a(P'',X,B+G)\le1$, a contradiction.
\end{proof}

\begin{lem}\label{lem:cbfbirinv}
Let $p$ be a positive integer, $(X,B)$ and $(X,B')$ two lc pairs, and $\phi:X\to T$ a contraction, such that $B'\ge B$, $K_X+B\sim_{\Rr,T}0$, and $K_X+B'\sim_{\Rr,T}0$. Let $B_T$ $($respectively $B_T')$ and $\bM_{\phi}$ be the discriminant part and a moduli part of the canonical bundle formula for $(X,B)$ $($respectively $(X,B'))$ over $T$. If $p(K_X+B)\sim p\phi^*(K_T+B_T+\bM_{\phi,T})$, then $p(K_X+B')\sim p\phi^*(K_T+B_T'+\bM_{\phi,T})$.
\end{lem}

\begin{proof}
Since $B'-B\ge0$ and $B'-B\sim_{\Rr,T}0$, $B'-B=\phi^*H_T$ for some $\Rr$-Cartier $\Rr$-divisor $H_T$ on $T$ by \cite[Lemma 2.4]{CHL22}. Then $B'_T=B_T+H_T$ by \cite[Lemma 7.4(ii)]{PS09}. Therefore
\begin{align*}
p(K_X+B')&=p(K_X+B)+p(B'-B)\sim p\phi^*(K_T+B_T+\bM_{\phi,T})+p\phi^*H_T\\&
=p\phi^*(K_T+B_T+H_T+\bM_{\phi,T})=p\phi^*(K_T+B_T'+\bM_{\phi,T}).
\end{align*}
\end{proof}

\begin{prop}\label{prop:cbfindex}
Let $\fR\subseteq[0,1]\cap\Qq$ be a finite set, and $\Phi:=\Phi(\fR)$. Then there exist a positive integer $p$ and a hyperstandard set $\Phi'\subseteq[0,1]\cap\Qq$ depending only on $\Phi$ satisfying the following.

Assume that $(X/Z,B)$ is an lc pair of dimension $\le3$ and $\phi:X\to T$ is a contraction over $Z$ such that $\dim T>0$, $B\in\Phi$, and $K_X+B\sim_{\Qq,T}0.$ Then we can choose a moduli part $\bM_\phi$ of the canonical bundle formula for $(X,B)$ over $T$, such that $B_T\in\Phi'$, $p\bM_\phi$ is b-Cartier, and
$$p(K_X+B)\sim p\phi^*(K_T+B_T+\bM_{\phi,T}),$$
where $B_T$ is the discriminant part of the canonical bundle formula for $(X,B)$ over $T$.
Moreover, if $\dim T=\dim X-1$, then $p\bM_\phi$ is base point free over $Z$.
\end{prop}
\begin{proof}
The result follows from \cite[Proposition 3.1]{CHL22} and \cite[Theorem 5.5]{FMX19}.
\end{proof}

\begin{prop}\label{prop:bsa}
Let $\Phi\subseteq[0,1]\cap\Qq$ be a hyperstandard set. Then there exists a positive integer $p$ depending only on $\Phi$ satisfying the following.

Assume that $(X/Z,B)$ is a klt threefold pair and $\phi:X\to T$ is a contraction over $Z$, such that $\dim T=1$, $B\in\Phi$, and $K_X+B\sim_{\Qq,Z}0$. Then we can choose a moduli part $\bM_\phi$ of the canonical bundle formula for $(X,B)$ over $T$, such that 
$$p(K_X+B)\sim p\phi^*(K_T+B_T+\bM_{\phi,T}),$$
and $p\bM_{\phi}$ is base point free over $Z$, where $B_T$ is the discriminant part of the canonical bundle formula for $(X,B)$ over $T$.
\end{prop}
\begin{proof}
If $\dim Z=1$, then $T=Z$. It follows that $p_1\bM_{\phi,T}$ is Cartier and thus $p_1\bM_{\phi,T}\sim_{Z}0$, where $p_1$ is given by Proposition \ref{prop:cbfindex} depending only on $\Phi$, and $\bM_\phi$ is a moduli part chosen as in Proposition \ref{prop:cbfindex}. So in what follows, we may assume that $\dim Z=0$, i.e., $Z$ is a point.

If $K_T\not\equiv0$, then $T=\Pp^1$. Let $\bM_\phi$ be a moduli part chosen as in Proposition \ref{prop:cbfindex}. Then $p_1\bM_{\phi,T}$ is base point free.

Now assume that $K_T\equiv0$ and in particular, $B_T=0$. Let $F$ be a general fiber of $X\to T$ and $K_F+B_F:=(K_X+B)|_F\sim_\Qq0$. According to \cite[Proposition 6.4 and Theorem 1.1]{CH21}, $r(K_F+B_F)\sim0$ for some positive integer $r$ depending only on $\Phi$. Then there exist a rational function $\alpha \in K(X)$ and an $\Rr$-Cartier $\Rr$-divisor $L$ on $T$ such that $K_X + B + \frac{1}{r}(\alpha) = \phi^\ast L$. Let $\bM_{\phi,T} = L-K_T-B_T$. Then $r(K_X+B)\sim r\phi^*(K_T+B_T+\bM_{\phi,T})$. Let $b_F$ be the second Betti number of a smooth model of the index one cover of $F$. By Lemma \ref{lem:dcckltsing}, there exists a positive real number $\epsilon$ which only depends on $\Phi$ such that $(F,B_F)$ is $\epsilon$-lc. If $B_F\neq0$, then $F$ belongs to a bounded family by \cite[Theorem 6.9]{Ale94}, and hence $b_F$ has an upper bound.
If $B_F=0$, then $K_F\sim_\Qq0$, and hence $b_F \le 22$ by the classification of surfaces. Therefore by \cite[Theorem 1.2]{Flo14}, there exists a positive integer $p_2$ depending only on $b_F$ such that $p_2\bM_{\phi,T}\sim0$.

We conclude that $p:=p_1p_2r$ has the required property.
\end{proof}

%%%%%%%%%%%%%%%%%%%%%%%%%%%%%%%%%%%%%%%%
%%%%%%%%%%%%%%%%%%%%%%%%%%%%%%%%%%%%%%%%
\subsection{Lifting complements}
Now we turn to the following technical statement on lifting complements via the canonical bundle formula. 

\begin{prop}\label{prop:cbflift1}
Let $p$ and $n$ be two positive integers such that $p\mid n$. Let $(X/Z,B)$ be an lc pair and $\phi:X\to T$ a contraction over $Z$ such that $\dim T>0$ and $K_X+B\sim_{\Rr,T}0$. Let $B_T$ and $\bM_{\phi}$ be the discriminant part and a moduli part of the canonical bundle formula for $(X,B)$ over $T$, such that $p(K_{X}+B)\sim p\phi^*(K_{T}+B_{T}+\bM_{\phi,T})$ and $p \bM_{\phi}$ is b-Cartier.
Let $(T',B_{T'}+\bM_{\phi}) \to (T,B_T+\bM_{\phi})$ be a crepant model and $M_{T'}$ an effective $\Qq$-divisor on $T'$, such that
\begin{enumerate}
  \item for any prime divisor $P\subseteq\Supp B$ which is vertical over $T$, the image of $P$ on $T'$ is a prime divisor,
  \item $pM_{T'}\sim_Z p\bM_{\phi,T'}$ and $M_{T'}\wedge B_{T'}=0$, and
  \item $(T'/Z\ni z,B_{T'}+M_{T'})$ is $n$-complementary for some $z\in Z$.
\end{enumerate}
Then $(X/Z\ni z,B)$ is also $n$-complementary.
\end{prop}
\begin{proof}
Let $X'$ be the normalization of the main component of $X\times_{T}T'$. Denote by $f:X'\to X$ and $\phi':X'\to T'$ the induced morphisms. We may write $K_{X'}+B'=f^*(K_X+B)$ for some $\Rr$-divisor $B'$. Note that by our assumption, we have
$$p(K_{X'}+B')\sim p\phi'^*(K_{T'}+B_{T'}+\bM_{\phi,T'})\sim_Z p\phi'^*(K_{T'}+B_{T'}+M_{T'}).$$
Let $(T'/Z\ni z,B_{T'}^++M_{T'})$ be an $n$-complement of $(T'/Z\ni z,B_{T'}+M_{T'})$. We remark that as $p\mid n$, $B^+_{T'}\ge 0$. Possibly shrinking $Z$ near $z$, we may assume that
$$n(K_{T'}+B_{T'}^++M_{T'})\sim_Z0.$$ 
Let $B'^+:=B'+\phi'^*(B_{T'}^+-B_{T'})$ and $B^+:=f_*B'^+$. We claim that $(X/Z\ni z,B^+)$ is an $n$-complement of $(X/Z\ni z,B)$. Indeed, we have
\begin{align*}
n(K_{X'}+B'^+)&=n(K_{X'}+B')+n(B'^+-B')\\&
\sim_Z n\phi'^*(K_{T'}+B_{T'}+M_{T'})+n\phi'^*(B_{T'}^+-B_{T'})\\& = n\phi'^*(K_{T'}+B^+_{T'}+M_{T'})\sim_Z0.
\end{align*}
Hence $n(K_X+B^+)\sim_Z0$. According to Lemma \ref{lem:cbfcptoverbase}(2), the sub-pair $(X'/Z\ni z,B'^{+})$ is sub-lc, and thus $(X/Z\ni z,B^{+})$ is also sub-lc. It suffices to prove that 
$$nB^+\ge \lf(n+1)\{B\}\rf+n\lf B\rf.$$

Let $P\subseteq\Supp B^+$ be a prime divisor. If $P$ is horizontal over $T$, then $\mult_PB^+=\mult_PB$ and there is nothing to prove. So we may assume that $Q$, the image of $P$ on $T'$, is a prime divisor. Let $b_P:=\mult_PB$, $b_P^+:=\mult_PB^+,b_Q:=\mult_QB_{T'}$, $b_Q^+:=\mult_QB_{T'}^+$, and $m_Q:=\mult_P\phi'^*Q$ over the generic point of $Q$. It is clear that $b_Q^+\ge0$ as $B_{T'}^+\ge0$. By construction,
$$b_P^+=b_P+(b_Q^+-b_Q)m_Q.$$
Hence
$$r_{PQ}:=b_P+(1-b_Q)m_Q=b_P^++(1-b_Q^+)m_Q\in\frac{1}{n}\Zz_{\ge0}.$$
Moreover, as $1-b_Q$ is the lc threshold of $\phi'^*Q$ with respect to $(X',B')$ over the generic point of $Q$, we know $r_{PQ}\le1$. If $b_Q=1$, then $b_Q^+=1$ and thus $b_P=b_P^+.$ If $b_P=1$, then $r_{PQ}=b_Q=1$ and thus $b_P^+=1$. Hence we may assume that $b_Q<1$ and $b_P<1$. 
Since $b_P=b_P^+- (b_Q^+-b_Q)m_Q$ and $nb_Q^+\ge\lf(n+1)b_Q\rf,$ we can see that
\begin{align*}
\lf(n+1)b_P\rf&=\lf(n+1)b_P^++(n+1)(b_Q-b_Q^+)m_Q\rf
%\\&=\bigg\lf(n+1)\bigg(r_{PQ}-(1-b_Q^+)m_Q+(b_Q-b_Q^+)m_Q\bigg)\bigg\rf
%\\&=\lf (n+1)b_P^++(n+1)(b_Q-b_Q^+)m_Q\rf
\\&=nb_P^++\lf b_P^++((n+1)b_Q-nb_Q^+)m_Q-b_Q^+m_Q\rf
\\&\le nb_P^++\lf b_P^++\{(n+1)b_Q\}m_Q-b_Q^+m_Q\rf
\\&=nb_P^+,
\end{align*}
where the last equality holds as
$$b_P^++\{(n+1)b_Q\}m_Q-b_Q^+m_Q<b_P^++m_Q-b_Q^+m_Q=r_{PQ}\le1.$$
We finish the proof.
\end{proof}

\section{Boundedness of complements for sdlt curves}\label{sec4}
\begin{defn}
We say $X$ is a \emph{semismooth curve} if $X$ is a reduced scheme of dimension $1$, every irreducible component of $X$ is normal, and all of its singularities are simple normal crossing points. 

Let $X$ be a semismooth curve, and let $B\ge0$ be an $\Rr$-divisor on $X$. We say $(X,B)$ is \emph{sdlt} if $B$ is supported in the smooth locus of $X$ and $\rddown{B}\le 1$.
\end{defn}

\begin{defn}
Let $X$ be a semismooth curve, and $B\ge0$ an $\Rr$-divisor on $X$, such that $(X,B)$ is sdlt.  
We say that $(X,B^+)$ is an \emph{$n$-semi-complement} of $(X,B)$, if
\begin{enumerate} 
  \item $(X,B^+)$ is sdlt,
  \item $nB^+\ge n\lfloor B\rfloor+\lfloor(n+1)\{B\}\rfloor$, and
  \item $n(K_X+B^+)\sim 0$.    
\end{enumerate} 
Moreover, we say $(X,B^+)$ is \emph{monotonic} if we additionally have $B^{+}\ge B$.
\end{defn}

The following theorem is a generalization of \cite[5.2.2]{Sho92} and \cite[19.4 Theorem]{Kol92} where the case $l=1$ is proved. 

\begin{thm}\label{thm:sdltcurvescpt}
Let $l$ be a positive integer. Then there exists a finite set of positive integers $\cN_{sdlt}$ divisible by $l$ depending only on $l$ satisfying the following.

Assume that $X$ is a semismooth curve, connected but not necessarily complete, and $B\ge0$ is an $\Rr$-divisor on $X$, such that
 \begin{enumerate}
     \item $(X,B)$ is sdlt,
     \item $X$ has at least one complete component,
     \item each incomplete component of $X$ does not meet any other incomplete component of $X$,
     \item the union of the complete components of $X$ is connected, and
     \item $-(K_X+B)$ is nef on each complete component of $X$.
 \end{enumerate}
 Then there exists an $n$-semi-complement $(X,B^+)$ of $(X,B)$ in a neighborhood of the union of the complete components of $X$ for some $n\in \cN_{sdlt}$.
\end{thm}

\begin{proof}
Let $X_0$ be a complete component of $X$, and let $\{P_1,\ldots,P_k\}:=X_0\cap \Sing X$. Then $\deg(K_X|_{X_0})=2g-2+k$, where $g$ is the genus of $X_0$. Since $\deg(K_X|_{X_0})\le0,$ there are four possibilities:
\begin{enumerate}[label=(\roman*)]
    \item $g=1,k=0,\deg(K_X|_{X_0})=0$,
    \item $g=0,k=2,\deg(K_X|_{X_0})=0$,
    \item $g=0,k=1,\deg(K_X|_{X_0})=-1$,
    \item $g=0,k=0,\deg(K_X|_{X_0})=-2$.
\end{enumerate}
We remark that $B$ could not meet the components of type (\romannum{1}) or (\romannum{2}) as $\deg(K_X|_{X_0})=0$. 

If $X_0$ is of type (\romannum{1}), then $X=X_0$ and $B=0$. In this case, $(X,B)$ is $l$-complementary. 

If $X_0$ is of type (\romannum{4}), then $X=X_0$ and $X\cong \Pp^1$. By Theorem \ref{thm:curcmpt}, there exists a finite set of positive integers $\cN'$ divisible by $l$ depending only on $l$, such that $(X,B)$ is $\cN'$-complementary.

Now suppose that any complete component of $X$ is either of type (\romannum{2}) or of type (\romannum{3}). Note that each component of type (\romannum{2}) (respectively type (\romannum{3})) can only meet other components at two points (respectively one point). By assumptions (3) and (4), the entire curve $X$ must form a chain or a cycle.
If $X$ is a cycle, then $B=0$ and $K_X \sim 0$ by Lemma \ref{lem: curve cpt kx of a cycle}.
Otherwise, by Lemma \ref{lem: curve cpt glue through a chain}, it suffices to construct $B^+$ such that $(X, B^+)$ is an $n$-complement of $(X, B)$ on each component of $X$.
Note that possibly shrinking $X$ near the union of the complete components, for any positive integer $n$, $(X,B)$ is an $n$-complement of itself on each incomplete component and each complete component of type (\romannum{2}).
Since $X$ has at most two complete components of type (\romannum{3}), by Lemma \ref{lem: bdd of complement for 2 curves} there exists a finite set of positive integers $\cN''$ divisible by $l$ depending only on $l$, such that $(X,B)$ has an $n$-complement on each complete component of type (\romannum{3}) for some $n\in\cN''$.
Therefore by Lemma \ref{lem: curve cpt glue through a chain}, $(X,B)$ has an $n$-semi-complement for some $n\in\cN''$.

Let $\cN_{sdlt}:=\cN'\cup\cN''$ and we are done.
\end{proof}

\begin{lem}\label{lem: curve cpt kx of a cycle}
Let $X = \bigcup_{i=1}^m X_i$ be a semismooth curve which is a cycle of irreducible curves $X_i$. Suppose that $X_i \cong \Pp^1$ for any $1 \le i \le m$. Then $K_X \sim 0$.
\end{lem}

\begin{proof}
For each integer $m\ge j \ge 2$, we construct a semismooth curve $Y_j$ in a smooth projective surface $S_{j}$ such that $Y_j$ is a cycle of $j$ complete rational curves and $K_{S_j}+Y_j \sim 0$, in particular, $K_{Y_j} = (K_{S_j} + Y_j)|_{Y_j} \sim 0$. Let $Y_2 \subseteq \Pp^2 =: S_2$ be the union of a line and a conic which is semismooth. Then $K_{S_2} + Y_2 \sim 0$ and thus $K_{Y_2} = (K_{S_2} + Y_2)|_{Y_2} \sim 0$. Suppose that we have constructed a semismooth curve $Y_{j-1}$ contained in a smooth projective surface $S_{j-1}$, such that $Y_{j-1}$ is a cycle of $j-1$ complete rational curves and $K_{S_{j-1}} + Y_{j-1} \sim 0$. Let $\pi_j: S_j \to S_{j-1}$ be the blow-up of $S_{j-1}$ at one snc point of $Y_{j-1}$, and $E_j$ the exceptional divisor of $\pi_j$. Let $Y_j = (\pi_j)_\ast^{-1} Y_{j-1} \cup E_j$. Then we get a semismooth curve $Y_j \subseteq S_j$, which is a cycle of $j$ complete rational curves, such that $K_{S_j} + Y_j \sim 0$. Since $X$ is analytically isomorphic to $Y_m$, by \cite[Appendix B, Theorem 2.1]{GTM52}, $K_{Y_m}\sim 0$ implies $K_X \sim 0$.
\end{proof}

\begin{lem}\label{lem: curve cpt glue through a chain}
Let $X = \bigcup_{i=1}^m X_i$ be a semismooth curve which is a chain of irreducible curves $X_i$. Suppose that $D$ is an $\Rr$-divisor on $X$, supported in the smooth locus of $X$, such that $D|_{X_i} \sim 0$ for any $1 \le i \le m$. Then $D \sim 0$.
\end{lem}

\begin{proof}
Let $X^{(i)} := \bigcup_{j=1}^i X_j$ for $1 \le i \le m$, and $P_i := X_i \cap X_{i+1} = X^{(i)} \cap X_{i+1}$ for $1 \le i \le m-1$. We will prove by induction that $D|_{X^{(i)}} \sim 0$ for any $1\le i\le m$. Suppose that $D|_{X^{(i-1)}} \sim 0$ for some integer $i\ge2$. Then there exist a rational function $\alpha_{i-1}$ on $X^{(i-1)}$ and a rational function $\beta_i$ on $X_i$, such that $D|_{X^{(i-1)}} = (\alpha_{i-1})$ and $D|_{X_i} = (\beta_i)$. Since $P_{i-1}$ is not contained in the support of $D$, $\alpha_{i-1}$ and $\beta_i$ are non-zero regular functions near $P_{i-1}$. Replacing $\beta_i$ by $\frac{\alpha_{i-1}(P_{i-1})}{\beta_i(P_{i-1})} \beta_i$, we may assume that $\alpha_{i-1}(P_{i-1}) = \beta_i(P_{i-1})$. Then there exists a rational function $\alpha_i$ on $X^{(i)}$ such that $\alpha_i|_{X^{(i-1)}} = \alpha_{i-1}$ and $\alpha_i|_{X_i} = \beta_i$. Hence $D|_{X^{(i)}} = (\alpha_i)$, and thus $D|_{X^{(i)}} \sim 0$. Therefore by induction we see that $D \sim 0$.
\end{proof}

\begin{lem}\label{lem: bdd of complement for 2 curves}
Let $l$ be a positive integer. Then there exists a finite set of positive integers $\cN''$ divisible by $l$ depending only on $l$ satisfying the following.

Assume that $\{a_i\}_{i=1}^{k}$ and $\{b_i\}_{i=1}^{k}$ are two sequences of non-negative real numbers, such that $\sum_{i=1}^{k} a_i\le 1$ and $\sum_{i=1}^{k} b_i\le 1$. Then there exist positive integers $n\in\cN''$ and $k'\ge k$, and two sequences of non-negative real numbers $\{a_i^{+}\}_{i=1}^{k'}$ and $\{b_i^{+}\}_{i=1}^{k'}$, such that
\begin{enumerate}
  %\item $p\mid n$, 
  \item $\sum_{i=1}^{k'} a_i^{+}=\sum_{i=1}^{k'} b_i^{+}=1$, and
  \item $na_i^{+}\ge n\rddown{ a_i}+\rddown{(n+1)\{a_i\}}$ and $nb_i^{+}\ge n\rddown{ b_i}+\rddown{(n+1)\{b_i\}}$ for any $1\le i\le k$.
\end{enumerate}
\end{lem}

\begin{proof}
Without loss of generality, we may assume that $a_i,b_i<1$ for any $i$. Then it suffices to prove
\begin{equation}\label{eq: bdd of complement for 2 curves}
    n-\sum_{i=1}^k\rddown{(n+1)a_i}\ge0\text{ and }n-\sum_{i=1}^k\rddown{(n+1)b_i}\ge0.
\end{equation}

For any positive integer $n$ and non-negative real numbers $c,d$, we have
$$\lfloor(n+1)(c+d)\rfloor\ge \lfloor(n+1)c\rfloor+ \lfloor(n+1)d\rfloor.$$ 
Thus possibly replacing $(a_i,a_j)$ by $(a_i+a_j,0)$ (respectively $(b_i,b_j)$ by $(b_i+b_j,0)$), we may assume that $a_i+a_j\ge 1$ (respectively $b_i+b_j\ge 1$) for any $i\neq j$. In particular, we may assume that $k=2$ and $a_1+a_2=b_1+b_2=1$.

By Dirichlet prime number theorem, there exist three distinct prime numbers $q_j$ such that $l\mid q_j-1$ for any $j\in\{1,2,3\}$. Let $n_j:=q_j-1$, and  $\cN'':=\{n_1,n_2,n_3\}$. We claim that there exists $n\in\cN''$ satisfying \eqref{eq: bdd of complement for 2 curves}. It suffices to show that both $(n_j+1)a_1$ and $(n_j+1)b_1$ are not integers for some $j\in\{1,2,3\}$. Otherwise, by the pigeonhole principle, we may assume that $a_1\in \frac{1}{n_j+1}\Zz\cap[0,1)=\frac{1}{q_j}\Zz\cap[0,1)$ for two indices $j\in\{1,2,3\}$, which is absurd.
\end{proof}

\begin{prop}\label{prop: bdd of comp of sdlt curves of dlt surfaces}
Let $l$ be a positive integer. Then there exists a finite set of positive integers $\cN$ divisible by $l$ depending only on $l$ satisfying the following.

Assume that $(X/Z\ni z,B)$ is a surface pair such that $z$ is a closed point, $(X,B)$ is dlt, $S:=\rddown{B}\neq 0$, $B-S$ is big over $Z$ and $K_X+B\sim_{\Rr,Z}0$. Then over a neighborhood of $z$, and $(S,B_S)$ has an $n$-semi-complement for some $n\in\cN$, where $K_S+B_S:=(K_X+B)|_{S}$.
\end{prop}
\begin{proof}
	Let $\cN_{sdlt}$ and $\cN_1$ be finite sets of positive integers divisible by $l$ given by Theorem \ref{thm:sdltcurvescpt} and Theorem \ref{thm:curcmpt} respectively which only depend on $l$. We will show that $\cN:=\cN_{sdlt}\cup\cN_1$ has the required property. 
	
	It is clear that $S$ is a semismooth curve, and $(S,B_S)$ is sdlt. We first show that $S$ is connected over a neighborhood of $z$. Otherwise, there exists a contraction $\phi:X\to T$ to a curve $T$ such that the general fiber $F$ of $\phi$ is $\Pp^1$ and each connected component of $S$ is horizontal over $T$; see Shokurov's connectedness lemma \cite[17.4 Theorem]{Kol92} and \cite[Propositions 3.3.1 and 3.3.2]{Pro01} (see also \cite[5.7 Connectedness lemma]{Sho92}, \cite[Corollary 1.3]{HH19}). Note that if $\dim Z=1$, then we take $T=Z$. As $B-S$ is big over $Z$, $B-S$ is horizontal over $T$ and $(B-S)|_F\neq\emptyset$. It follows that $(K_X+B)|_F=(K_X+S+B-S)|_F\not\sim_\Rr0$, a contradiction. Thus $S$ is connected over a neighborhood of $z$. Possibly shrinking $Z$ near $z$, we may assume that $K_X+B\sim_\Rr0$ and thus $K_S+B_S$ is trivial on each complete component of $S$. If $S$ has two irreducible incomplete components $S_1$ and $S_2$ that $S_1\cap S_2\neq\emptyset$ over any neighborhood of $z$, then by assumption, we have $B=S=S_1+S_2$ over a neighborhood of $z$. In this case, $B_S=0$ and $K_S\sim0$ over a neighborhood of $z$. Now we assume that each irreducible incomplete component of $S$ does not meet any other irreducible incomplete component of $S$. By the classification of dlt surface pairs (cf. \cite[Corollary 5.55]{KM98}), over a neighborhood of $z$, either the support of $B_S$ lies in the union of the complete components of $S$ or $S$ is irreducible and its image on $Z$ is also a curve. In the former case, $(S,B_S)$ has an $n$-semi-complement in a neighborhood of the union of complete component of $S$ for some $n\in\cN_{sdlt}$ by Theorem \ref{thm:sdltcurvescpt}. Therefore over a neighborhood of $z$, $(S,B_S)$ has an $n$-semi-complement. In the latter case, the morphism from $S$ to its image on $Z$ is a contraction, then $(S,B_S)$ has an $n$-complement over a neighborhood of $z$ for some $n\in\cN_1$. This finishes the proof.
\end{proof}

\section{Boundedness of complements for surfaces}\label{sec5}
\subsection{Conjecture \ref{conj:bddcpt2} for surfaces}
In this subsection, we confirm Conjecture \ref{conj:bddcpt2} for surfaces. For convenience, by (Theorem $*$)$_d$ we mean Theorem $*$ in dimension $d$.

\medskip
\noindent{\bf Notation $(\star)$.}\label{notation}
Let $\Phi_1\subseteq[0,1]\cap\Qq$ be a hyperstandard set. Let $p=p(\Phi_1)$ be a positive integer given by (Proposition \ref{prop:cbfindex})$_2$ which only depends on $\Phi_1$. Let $\cN_2=\cN_2(p)$ be a finite set of positive integers divisible by $p$ given by Theorem \ref{thm:curcmpt} which only depends on $p$, and let $\Phi_2:=\Gamma(\cN_2,\Phi_1)$.

\begin{prop}\label{prop:21case}
Under {\bf Notation $(\star)$}, assume that $(X/Z,B)$ is a $\Qq$-factorial lc surface pair such that $K_X+B\sim_{\Rr,Z}0$ and $\kappa(X/Z,B-B_{\Phi_2})+\dim Z=\kappa(X/Z,B-B_{\Phi_1})+\dim Z=1$. Then $(X/Z\ni z,B)$ is $\cN_2$-complementary for any closed point $z\in Z$.
\end{prop}
\begin{proof}
By Lemma \ref{lem:runantiMMP} we can run an MMP on $-(K_X+B_{\Phi_2})\sim_{\Rr,Z}B-B_{\Phi_2}$ over $Z$ and reach a good minimal model $\psi:X\to X'$ over $Z$, such that $-(K_{X'}+B_{\Phi_2}')$ is semi-ample over $Z$, where $D'$ denotes the strict transform of $D$ on $X'$ for any $\Rr$-divisor $D$ on $X$. Let $\pi':X'\to Z'$ be the contraction defined by $-(K_{X'}+B_{\Phi_2}')$ over $Z$. By assumption, $\dim Z'=1$. Let $B_{Z'}^{(2)}$ and $\bM_{\pi'}$ be the discriminant and moduli parts of the canonical bundle formula for $(X',B_{\Phi_{2}}')$ over $Z'$ in Proposition \ref{prop:cbfindex}.
\begin{claim}\label{claim:lift1}
$p\bM_{\pi'}$ is base point free over $Z$ and 
$$p\left(K_{X'}+B_{\Phi_{2}}'\right)\sim p(\pi')^*\left(K_{Z'}+B_{Z'}^{(2)}+\bM_{\pi',Z'}\right).$$
\end{claim}

Assume Claim \ref{claim:lift1}. Then
$$p\left(K_{X}+B^{(2)}\right):=p\psi^*\left(K_{X'}+B'_{\Phi_2}\right)\sim p\left(\pi'\circ\psi\right)^{*}\left(K_{Z'}+B_{Z'}^{(2)}+\bM_{\pi',Z'}\right).$$
Note that since $\psi$ is $-(K_X+B_{\Phi_2})$-negative, $B^{(2)}\ge B_{\Phi_2}$. Since $K_{X'}+B'\sim_{\Rr,Z}0$ and $B\ge B_{\Phi_2}$, there exists a boundary $B_{Z'}$ on $Z'$ such that the g-pair $(Z',B_{Z'}+\bM_{\pi'})$ is glc, $B_{Z'}\ge B_{Z'}^{(2)}$, and $K_{Z'}+B_{Z'}+\bM_{\pi',Z'}\sim_{\Rr,Z}0$. As $p\bM_{\pi'}$ is base point free over $Z$, we can pick an effective $\Qq$-divisor $M_{Z'}$ on $Z'$ such that 
$$pM_{Z'}\sim_Z p\bM_{\pi',Z'}, M_{Z'}\wedge B_{Z'}=0\text{, and }(Z',B_{Z'}+M_{Z'})\text{ is lc}.$$ 
In particular, $(Z'/Z\ni z,B_{Z'}+M_{Z'})$ is an $\Rr$-complement of $(Z'/Z\ni z,B_{Z'}^{(2)}+M_{Z'})$ for any $z\in Z$. Now by our choice of $\cN_2$, $(Z'/Z\ni z,B_{Z'}^{(2)}+M_{Z'})$ is $\cN_2$-complementary. According to Proposition \ref{prop:cbflift1}, $(X/Z\ni z,B^{(2)})$ is $\cN_2$-complementary, and hence $(X/Z\ni z,B_{\Phi_2})$ is also $\cN_2$-complementary as $B^{(2)}\ge B_{\Phi_2}$. Thus $(X/Z\ni z,B)$ is $\cN_2$-complementary by Lemma \ref{lem:N_Phicompl}. Therefore it suffices to proof Claim \ref{claim:lift1}.

\begin{proof}[Proof of Claim \ref{claim:lift1}]
According to Lemma \ref{lem:runantiMMP} again, we may run an MMP on $-(K_{X'}+B'_{\Phi_1})\sim_{\Rr,Z'}B'_{\Phi_{2}}-B'_{\Phi_{1}}$ over $Z'$ and reach a good minimal model $X'\to X''$ over $Z'$, such that $B''_{\Phi_{2}}-B''_{\Phi_{1}}$ is semi-ample over $Z'$, where $D''$ denotes the strict transform of $D'$ on $X''$ for any $\Rr$-divisor $D'$ on $X'$. One can pick a positive real number $\epsilon$, such that $g:X\to X''$ is also an MMP on $B-B_{\Phi_{2}}+\epsilon (B_{\Phi_{2}}-B_{\Phi_{1}})$ over $Z$. Furthermore, we may assume that $B''-B_{\Phi_{2}}''+\epsilon (B_{\Phi_{2}}''-B_{\Phi_{1}}'')$ is semi-ample over $Z$ by Lemma \ref{lem:relamp}.
\begin{center}
	\begin{tikzcd}[column sep = 2em, row sep = 2em]
		X \arrow[d,"",swap]\arrow[rr, "",]  && X' \arrow[d, "\pi'" swap] \arrow[rr, ] && X'' \arrow[d, "{\pi''}" swap] \\
		Z && \arrow[ll, ""] Z'  && \arrow[ll, "id", swap]Z'
	\end{tikzcd}
\end{center}
By assumption, 
$$\kappa(X/Z,B-B_{\Phi_{2}})=\kappa(X/Z,B-B_{\Phi_{2}}+\epsilon (B_{\Phi_{2}}-B_{\Phi_{1}})),$$
and $B''-B''_{\Phi_{2}}+\epsilon (B''_{\Phi_{2}}-B''_{\Phi_{1}})\sim_{\Rr,Z'}\epsilon(B''_{\Phi_{2}}-B^{''}_{\Phi_{1}})$. Hence the natural morphism $\pi'':X''\to Z'$ is the contraction defined by $B''_{\Phi_{2}}-B^{''}_{\Phi_{1}}$ over $Z'$. In particular, we have 
$$K_{X''}+B''_{\Phi_{1}}\sim_{\Rr,Z'}0\text{, and }K_{X''}+B''_{\Phi_{2}}\sim_{\Rr,Z'}0.$$

By Lemma \ref{lem:cbfbirinv} and Proposition \ref{prop:cbfindex}, we see that $p\bM_{\pi'}$ is base point free, and 
$$p(K_{X''}+B_{\Phi_2}'')\sim p(\pi'')^{*}\left(K_{Z'}+B_{Z'}^{(2)}+\bM_{\pi',Z'}\right).$$
Since $X' \to X''$ is $(K_{X'}+B_{\Phi_2}')$-trivial, $(X',B_{\Phi_2}')$ and $(X'',B_{\Phi_2}'')$ are crepant. Therefore
$$p(K_{X'}+B_{\Phi_2}')\sim p(\pi')^{*}\left(K_{Z'}+B_{Z'}^{(2)}+\bM_{\pi',Z'}\right).$$
We complete the proof.
\end{proof}
\end{proof}

\begin{thm}\label{thm:surfcmpt2}
Let $l$ be a positive integer and $\Phi\subseteq[0,1]\cap\Qq$ a hyperstandard set. Then there exists a finite set of positive integers $\cN$ divisible by $l$ depending only on $l$ and $\Phi$ satisfying the following. 

Assume that $(X/Z\ni z,B)$ is a surface pair such that $(X/Z\ni z,B_{\cN\_\Phi})$ has a klt $\Rr$-complement. Then $(X/Z\ni z,B)$ is $\cN$-complementary.
\end{thm}

\begin{proof}
Let $\cN_1=\cN_1(l,\Phi)$ be a finite set of positive integers divisible by $l$ given by (Theorem \ref{thm:ftbdd})$_2$ which only depends on $l$ and $\Phi$, and let $\Phi_1:=\Ii(\cN_1,\Phi)$. Let $p=p(l,\Phi_1)$ be a positive integer divisible by $l$ given by (Proposition \ref{prop:cbfindex})$_2$ which only depends on $l$ and $\Phi_1$. Let $\cN_2=\cN_2(p)$ be a finite set of positive integers divisible by $p$ given by Theorem \ref{thm:curcmpt} which only depends on $p$, and let $\Phi_2:=\Gamma(\cN_1\cup\cN_2,\Phi)$. Let $n_{CY}=n_{CY}(l, \Phi_2)$ be a positive integer divisible by $l$ given by Lemma \ref{lem:global} which only depends on $l$ and $\Phi_2$. We will show that the finite set $\cN:=\cN_1\cup\cN_2\cup\{n_{CY}\}$ has the required property.

Possibly replacing $z$ by a closed point of $\bar{z}$, we may assume that $z$ is a closed point. Suppose that $(X/Z\ni z,B^+)$ is a klt $\Rr$-complement of $(X/Z\ni z,B_{\cN\_\Phi})$. Possibly replacing $(X,B)$ by a small $\Qq$-factorialization of $(X,B^+)$ and shrinking $Z$ near $z$, we may assume that $(X,B)$ is $\Qq$-factorial klt and $K_X+B\sim_{\Rr,Z}0$. Since $B\ge B_{\Phi_2}\ge B_{\Phi_1}$,
$$0\le\kappa(X/Z,B-B_{\Phi_2})+\dim Z\le\kappa(X/Z,B-B_{\Phi_1})+\dim Z\le2.$$
Therefore, we only need to consider the following three cases:
\begin{enumerate}
  \item $\kappa(X/Z,B-B_{\Phi_1})+\dim Z=2$,
  \item $\kappa(X/Z,B-B_{\Phi_2})+\dim Z=\kappa(X/Z,B-B_{\Phi_1})+\dim Z=1$, and
  \item $\kappa(X/Z,B-B_{\Phi_2})+\dim Z=0$.
\end{enumerate}

If $\kappa(X/Z,B-B_{\Phi_1})+\dim Z=2$, then $X$ is of Fano type over $Z$. In this case $(X/Z\ni z,B)$ is $\cN_1$-complementary by the choice of $\cN_1$ (see Theorem \ref{thm:ftbdd}). If $\kappa(X/Z,B-B_{\Phi_2})+\dim Z=\kappa(X/Z,B-B_{\Phi_1})+\dim Z=1$, then $(X/Z\ni z,B)$ is $\cN_2$-complementary by Proposition \ref{prop:21case}. If $\kappa(X/Z,B-B_{\Phi_2})+\dim Z=0$, that is, $\dim Z=0$ and $\kappa(X,B-B_{\Phi_2})=0$, then one has
$$n_{CY}(K_X+B)\sim0$$
by the choice of $n_{CY}$ (see Lemma \ref{lem:global}). We finish the proof.
\end{proof}

\subsection{Proof of Theorem \ref{thm:surfcmpt}}

\begin{prop}[cf. {\cite[16.7 Corollary]{Kol92}}]\label{prop:adjcoeff}
Let $\fR\subseteq[0,1]\cap\Qq$ be a finite set and $\Phi:=\Phi(\fR)$. Then there exists a hyperstandard set $\tilde{\Phi}$ depending only on $\Phi$ satisfying the following.

Assume that $(X,B)$ is a dlt pair and $S:=\lf B\rf$. Let $K_S+B_S:=(K_X+B)|_S$. If $B\in\Phi$, then $B_S\in\tilde{\Phi}$, and if $B\in\Gamma(\{n\},\Phi)$ for some positive integer $n$, then $B_S\in\Ii(\{n\},\tilde{\Phi})$.
\end{prop}
\begin{proof}
Let
$$\tilde{\fR}:=\left\{1-\sum(1-r_i)\ge0\mid r_i\in\fR\right\},\ \fR_1:=\left\{r-\frac{m}{n+1}\ge0\mid r\in\fR,m\in\Zz_{\ge0}\right\},$$
and
\begin{align*}
\tilde{\fR}_1:&=\left\{1-\sum(1-r_i')\ge0\mid r_i'\in\fR_1\right\}
\\&=\left\{1-\sum(1-r_i)-\frac{m}{n+1}\ge0\mid r_i\in\fR,m\in\Zz_{\ge0}\right\}.
\end{align*}
Let $\tilde{\Phi}:=\Phi(\tilde{\fR})$ and $\tilde{\Phi}_1:=\Phi(\tilde\fR_1)$. It is clear that $\Phi(\fR_1)=\Ii(\{n\},\Phi)$ and $\tilde{\Phi}_1=\Ii(\{n\},\tilde{\Phi})$.
By \cite[16.7 Corollary]{Kol92}, if $B\in\Phi$, then $B_{S}\in\tilde{\Phi}$, and if $B\in\Gamma(\{n\},\Phi)$, then $B_S\in\Ii(\{n\},\tilde{\Phi})$. Therefore $\tilde{\Phi}$ has the required property.
\end{proof}

\begin{prop}\label{prop:adjlift}
Let $(X/Z\ni z, B)$ be a surface pair such that $(X,B)$ is $\Qq$-factorial dlt and $-(K_{X}+B)$ is nef and big over a neighborhood of $z$. Let $S:=\lfloor B\rfloor$ and  $K_{S}+B_{S}:=(K_X+B)|_S$. Suppose that $S$ intersects $X_z$, the fiber of $X\to Z$ over $z$, and $(S,B_{S})$ has a monotonic $n$-semi-complement $(S,B_{S}^{+})$ over a neighborhood of $z$. Then $(X/Z\ni z, B)$ is $n$-complementary.
\end{prop}

\begin{proof}
Possibly replacing $z$ by a closed point of $\bar{z}$ and shrinking $Z$ near $z$, we may assume that $z$ is a closed point, $(X,B)$ is $\Qq$-factorial dlt, $-(K_{X}+B)$ is nef and big over $Z$, and $n(K_S+B_S^+)\sim0$.

Let $g:W\to X$ be a log resolution of $(X,B)$ such that $g$ is an isomorphism over the snc locus of $(X,B)$ (cf. \cite[Theorem 10.45]{Kol13}), and let $S_W$ be the strict transform of $S$ on $W$. Then the induced morphism $g_{S_W}:=g|_{S_W}: S_W\to S$ is an isomorphism. We define
$$K_W+B_W:=g^{*}(K_X+B),\ n\left(K_{S_W}+B_{S_W}^{+}\right):=g_{S_W}^{*}\left(n\left(K_S+B_S^{+}\right)\right)\sim 0,$$
and 
$$L_W:=\lceil -(n+1)(K_W+B_W)\rceil.$$
Let $\Delta_W:=B_W-S_W$. Then
$$K_{S_W}+B_{S_W}:=(K_W+B_W)|_{S_W}=K_{S_W}+\Delta_W|_{S_W}=g_{S_W}^{*}(K_S+B_S),$$
and $B_{S_W}<1$ as $\Delta_W<1$.
	
Since $-(n+1)(K_W+B_W)$ is nef and big over $Z$,  $R^1h_{*}(\mathcal{O}_{W}(K_W+L_W))=0$ by the relative Kawamata-Viehweg vanishing theorem for $\Rr$-divisors (cf. \cite[Theorem 3.2.9]{Fuj17}), where $h$ is the induced morphism $W\to Z$. From the exact sequence
$$0\to\mathcal{O}_W(K_W+L_W)\to \mathcal{O}_W(K_W+S_W+L_W)\to \mathcal{O}_{S_W}(K_{S_W}+L_W|_{S_W})\to 0,$$
we deduce that the induced map
$$H^0(W,K_W+S_W+L_W)\rightarrow H^0(S_W,K_{S_W}+L_W|_{S_W})$$
is surjective. Since $nB_{S_W}^{+}\in\Zz$, $B_{S_W}<1$, and $B_{S_W}^{+}-B_{S_W}\ge0$, we see that 
$$G_{S_W}:=nB_{S_W}^{+}-\rddown{(n+1)B_{S_W}}$$ 
is an effective integral divisor. We have
\begin{align*}
K_{S_W}+L_W|_{S_W}&=K_{S_W}+\lceil -(n+1)(K_{S_W}+B_{S_W})\rceil\\
&=-nK_{S_W}-\rddown{(n+1)B_{S_W}}\sim G_{S_W}\ge0.
\end{align*}
Thus there exists $G_W\ge0$ on $W$ such that $G_W|_{S_W}=G_{S_W}$ and
$$G_W\sim K_W+S_W+L_W= -nK_W-nS_W-\lf(n+1)\Delta_W\rf.$$ 
Let $G:=g_{*}G_W$, and 
$$B^{+}:=S+\frac{1}{n}(\rddown{(n+1)\{B\}}+G).$$ 
Then we have 
$$n(K_X+B^{+})= n(K_X+S)+\rddown{(n+1)\{B\}}+G\sim 0.$$

It remains to show that $(X,B^{+})$ is lc over a neighborhood of $z$. Let $V$ be the non-lc locus of $(X,B^{+})$. There exists a real number $a\in(0,1)$, such that the non-klt locus of $(X,aB^{+}+(1-a)B)$ is equal to $S\cup V$.

Since $g^{*}(K_X+B^+)|_{S_W}={g_{S_W}}^{*}(K_S+B_S^+)$, we have $(K_X+B^{+})|_S=K_{S}+B_{S}^{+}$ and $(K_X+aB^{+}+(1-a)B)|_{S}=K_S+aB_S^{+}+(1-a)B_S$. By inversion of adjunction (cf. \cite[Theorem 1.4]{FH21}), $(X,aB^{+}+(1-a)B)$ is lc near $S$. In particular, $S$ is disjoint from $V$. Since 
$$-(K_X+aB^{+}+(1-a)B)=-a(K_X+B^{+})-(1-a)(K_X+B)$$
is nef and big over $Z$, by Shokurov-Koll\'ar connectedness principle (cf. \cite[17.4 Theorem]{Kol92}), $(S\cup  V)\cap X_z$ is connected. Recall that by assumption, $S\cap X_z\neq\emptyset$. Hence $V\cap X_z=\emptyset$ and $(X,B^{+})$ is lc over a neighborhood of $z$.
\end{proof}

\begin{thm}\label{thm:surfcmpt3}
Let $l$ be a positive integer and $\Phi\subseteq[0,1]\cap\Qq$ a hyperstandard set. Then there exists a finite set of positive integers $\cN$ divisible by $l$ depending only on $l$ and $\Phi$ satisfying the following. 

Assume that $(X/Z\ni z,B)$ is a surface pair such that $(X/Z\ni z,B_{\cN\_\Phi})$ is $\Rr$-complementary. Then $(X/Z\ni z,B)$ is $\cN$-complementary.
\end{thm}

\begin{proof}
Let $\tilde{\Phi}:=\Phi(\tilde{\fR})$ be the hyperstandard set associated to the finite set $\tilde{\fR}\subseteq[0,1]\cap\Qq$ given by Proposition \ref{prop:adjcoeff} which only depends on $\Phi$. Possibly replacing $l$ by a multiple, we may assume that $l\tilde{\fR}\subseteq\Zz_{\ge0}$. Let $\cN_0=\cN_0(l, \Phi)$ be a finite set of positive integers divisible by $l$ given by Theorem \ref{thm:surfcmpt2} which only depends on $l$ and $\Phi$. Let $\cN_1=\cN_1(l)$ be a finite set of positive integers divisible by $l$ given by Proposition \ref{prop: bdd of comp of sdlt curves of dlt surfaces} which only depends on $l$, and let $\Phi_1:=\Gamma(\cN_1,\Phi)$. Let $p=p(l, \Phi_1)$ be a positive integer divisible by $l$ given by (Proposition \ref{prop:cbfindex})$_2$ which only depends on $l$ and $\Phi_1$. Let $\cN_2=\cN_2(p)$ be a finite set of positive integers divisible by $p$ given by Theorem \ref{thm:curcmpt} which only depends on $p$, and let $\Phi_2:=\Gamma(\cN_1\cup\cN_2,\Phi)$. Let $n_{CY} = n_{CY}(l, \Phi_2)$ be a positive integer divisible by $l$ given by Lemma \ref{lem:global} which only depends on $l$ and $\Phi_2$. We will show that the finite set $\cN:=\cN_0\cup\cN_1\cup\cN_2\cup\{n_{CY}\}$ has the required property.

Possibly replacing $z$ by a closed point of $\bar{z}$, we may assume that $z$ is a closed point. If $(X/Z\ni z,B_{\cN\_\Phi})$ has a klt $\Rr$-complement, then so does $(X/Z\ni z,B_{{\cN_0}\_\Phi})$, and hence $(X/Z\ni z,B)$ is $\cN_0$-complementary by the choice of $\cN_0$. Therefore we may assume that $(X/Z\ni z,B_{\cN\_\Phi})$ has an $\Rr$-complement $(X/Z\ni z,B_{\cN\_\Phi}+G)$ which is not klt. Possibly replacing $(X/Z\ni z,B)$ by a $\Qq$-factorial dlt modification of $(X/Z\ni z,B_{\cN\_\Phi}+G)$ and shrinking $Z$ near $z$, we may assume that $(X,B)$ is $\Qq$-factorial dlt, $K_X+B\sim_{\Rr,Z}0$, and $S\cap X_z\neq\emptyset$, where $S:=\rddown{B}\neq 0$ and $X_z$ is the fiber of $X\to Z$ over $z$. Since $B\ge B_{\Phi_2}\ge B_{\Phi_1}\ge B_{\Phi}$, we have
\begin{align*}
0&\le\kappa(X/Z,B-B_{\Phi_2})+\dim Z\le\kappa(X/Z,B-B_{\Phi_1})+\dim Z
\\&\le\kappa(X/Z,B-B_{\Phi})+\dim Z\le2.
\end{align*}
Therefore, we only need to consider the following three possibilities:
\begin{enumerate}
  \item $\kappa(X/Z,B-B_{\Phi_2})+\dim Z=0$,
  \item $\kappa(X/Z,B-B_{\Phi_2})+\dim Z=\kappa(X/Z,B-B_{\Phi_1})+\dim Z=1$, and
  \item $\kappa(X/Z,B-B_{\Phi_1})+\dim Z=\kappa(X/Z,B-B_{\Phi})+\dim Z=2$.
\end{enumerate}

If $\kappa(X/Z,B-B_{\Phi_2})+\dim Z=0$, then $n_{CY}(K_X+B)\sim0$ by our choice of $n_{CY}$. If $\kappa(X/Z,B-B_{\Phi_2})+\dim Z=\kappa(X/Z,B-B_{\Phi_1})+\dim Z=1$, then $(X/Z\ni z,B)$ is $\cN_2$-complementary by Proposition \ref{prop:21case}. Hence in what follows we assume that $\kappa(X/Z,B-B_{\Phi_1})+\dim Z=\kappa(X/Z,B-B_{\Phi})+\dim Z=2$. We will show that $(X/Z\ni z,B)$ is $\cN_1$-complementary.

In this case, both $B-B_{\Phi}$ and $B-B_{\Phi_{1}}$ are big over $Z$. Let $K_S+B_S:=(K_X+B)|_{S}\sim_{\Rr,Z}0$. By Lemma \ref{lem:N_Phicompl} and the choice of $\cN_1$, $(S,(B_S)_{n\_\tilde{\Phi}})$ has a monotonic $n$-semi-complement over a neighborhood of $z$ for some $n\in \cN_1$. Note that $B_{n\_\Phi}\in\Gamma(\{n\},\Phi)\subseteq \Gamma(\cN_1,\Phi)$, $B_{n\_\Phi}\le B_{\Phi_{1}}$, and $B-B_{n\_\Phi}$ is big over $Z$. According to Lemma \ref{lem:runantiMMP}, we may run an MMP on $-(K_{X}+B_{n\_\Phi})\sim_{\Rr,Z}B-B_{n\_\Phi}$ over $Z$ and reach a minimal model $\psi:X\to X'$ over $Z$, such that $B'-B'_{n\_\Phi}$ is nef and big over $Z$, where $D'$ denotes the strict transform of $D$ on $X'$ for any $\Rr$-divisor $D$ on $X$. No component of $S$ is contracted by $\psi$ and $\psi_S := \psi|_{S}:S\to S'$ is an isomorphism as $S\le B_{n\_\Phi}\le B$ and $\psi$ is $(K_X+B)$-trivial.

Since $-(K_X+B)-\psi^{*}(-(K_{X'}+B'_{n\_\Phi}))$ is nef over $X'$, and $-B'+B'_{n\_\Phi}\le 0$, by the negativity lemma, $-(K_X+B)\le \psi^{*}(-(K_{X'}+B'_{n\_\Phi}))$. Let
$$K_{S'}+B'_{n\_\Phi,S'}:=\left(K_{X'}+B'_{n\_\Phi}\right)|_{S'}.$$ 
Note that $B'_{n\_\Phi,S'}\in\Gamma(\{n\},\tilde{\Phi})$ by Proposition \ref{prop:adjcoeff}, and the support of $-(K_X+B)- \psi^{*}(-(K_{X'}+B'_{n\_\Phi}))$ does not contain any component of $S$. Then
\begin{align*}
-(K_S+B_S)- \psi_{S}^{*}\left(-\left(K_{S'}+B'_{n\_\Phi,S'}\right)\right)
=\left(-(K_X+B)- \psi^{*}\left(-\left(K_{X'}+B'_{n\_\Phi}\right)\right)\right)|_{S}\le 0.
\end{align*}
Let $B_S'$ be the strict transform of $B_S$ on $S'$. Since $(B_S')_{n\_\tilde{\Phi}},B'_{n\_\Phi,S'}\in\Gamma(\{n\},\tilde{\Phi})$, and $B_S'\ge B'_{n\_\Phi,S'}$, we deduce that $(B'_S)_{n\_\tilde{\Phi}}\ge B'_{n\_\Phi,S'}$. Hence $(S',B'_{n\_\Phi,S'})$ has a monotonic $n$-semi-complement over a neighborhood of $z$. By Proposition \ref{prop:adjlift} and Lemma \ref{lem:N_Phicompl}, $(X'/Z\ni z,B'_{n\_\Phi})$ has a monotonic $n$-complement. Since $\psi$ is $-(K_X+B_{n\_\Phi})$-negative, $(X/Z\ni z,B_{n\_\Phi})$ has a monotonic $n$-complement $(X/Z\ni z,B^{+})$. By Lemma \ref{lem:N_Phicompl}, $(X/Z\ni z,B^{+})$ is an $n$-complement of $(X/Z\ni z,B)$. 
\end{proof}
\begin{proof}[Proof of Theorem \ref{thm:surfcmpt}]
Possibly replacing $(X/Z\ni z,B)$ with a dlt modification, we may assume $X$ is $\Qq$-factorial. Then the theorem follows by Theorem \ref{thm:surfcmpt3}.
\end{proof}

\section{Boundedness of complements for threefolds}\label{sec6}
We will prove the following theorem which is stronger than Theorem \ref{thm:3foldsbddcmpt}.
\begin{thm}\label{thm:3foldsbddcmpt2}
Let $l$ be a positive integer and $\Phi\subseteq[0,1]\cap\Qq$ a hyperstandard set. Then there exists a finite set of positive integers $\cN$ divisible by $l$ depending only on $l$ and $\Phi$ satisfying the following.

Assume that $(X/Z\ni z,B)$ is a threefold pair such that $(X/Z\ni z,B_{\cN\_\Phi})$ has a klt $\Rr$-complement. Then $(X/Z\ni z,B)$ is $\cN$-complementary.
\end{thm}

\begin{proof}
Let $\cN_1 = \cN_1(l, \Phi)$ be a finite set of positive integers divisible by $l$ given by (Theorem \ref{thm:ftbdd})$_3$ which only depends on $l$ and $\Phi$, and set $\Phi_1:=\Ii(\cN_1,\Phi)$. Let $p_1 = p_1(l, \Phi_1)$ be a positive integer divisible by $l$ given by (Proposition \ref{prop:cbfindex})$_3$ and Proposition \ref{prop:bsa} which only depends on $l$ and $\Phi_1$. Let $\cN_2 = \cN_2(p_1)$ be a finite set of positive integers divisible by $p_1$ given by Theorems \ref{thm:curcmpt} and \ref{thm:surfcmpt} which only depends on $p_1$, and set $\Phi_2:=\Gamma(\cN_1\cup\cN_2,\Phi)$. Let $p_2 = p_2(p_1, \Phi_2)$ be a positive integer divisible by $p_1$ given by (Proposition \ref{prop:cbfindex})$_3$ and Proposition \ref{prop:bsa} which only depends on $p_1$ and $\Phi_2$. Let $\cN_3 = \cN_3(p_2)$ be a finite set of positive integers divisible by $p_2$ given by Theorems \ref{thm:curcmpt} and \ref{thm:surfcmpt} which only depends on $p_2$, and set $\Phi_3:=\Gamma(\cN_1\cup\cN_2\cup\cN_3,\Phi)$. Let $n_{CY} = n_{CY}(l, \Phi_3)$ be a positive integer divisible by $l$ given by Lemma \ref{lem:global} which only depends on $l$ and $\Phi_3$. We will show that $\cN:=\cN_1\cup\cN_2\cup\cN_3\cup\{n_{CY}\}$ has the required property.

Replacing $z$ by a closed point of $\bar{z}$, we may assume that $z$ is a closed point. Possibly replacing $(X/Z\ni z,B)$ by a small $\Qq$-factorialization of a klt $\Rr$-complement of $(X/Z\ni z,B_{\cN\_\Phi})$ and shrinking $Z$ near $z$, we may assume that $(X,B)$ is $\Qq$-factorial klt and $K_X+B\sim_{\Rr,Z}0$.

If $\kappa(X/Z,B-B_{\Phi_1})+\dim Z=3$, then $X$ is of Fano type over $Z$. In this case $(X/Z\ni z,B)$ is $\cN_1$-complementary by the choice of $\cN_1$ (see Theorem \ref{thm:ftbdd}). If $\kappa(X/Z,B-B_{\Phi_3})+\dim Z=0$, then $n_{CY}(K_X+B)\sim0$ by the choice of $n_{CY}$ (see Lemma \ref{lem:global}). So in the following we may assume that 
\begin{align*}
1&\le \kappa(X/Z,B-B_{\Phi_3})+\dim Z\le\kappa(X/Z,B-B_{\Phi_2})+\dim Z\\ &\le\kappa(X/Z,B-B_{\Phi_1})+\dim Z\le2.
\end{align*}
In particular, there exist integers $i,k\in\{1,2\}$ such that
$$\kappa(X/Z,B-B_{\Phi_i})+\dim Z=\kappa(X/Z,B-B_{\Phi_{i+1}})+\dim Z=k.$$ 
We will show that $(X/Z\ni z,B_{\Phi_{i+1}})$ is $\cN_{i+1}$-complementary and thus finish the proof by Lemma \ref{lem:N_Phicompl}.

By Lemma \ref{lem:runantiMMP}, we can run an MMP on $-(K_X+B_{\Phi_{i+1}})\sim_{\Rr,Z}B-B_{\Phi_{i+1}}$ over $Z$ and reach a good minimal model $X\dashto X'$ over $Z$, such that $B'-B_{\Phi_{i+1}}'$ is semi-ample over $Z$, where $D'$ denotes the strict transform of $D$ on $X'$ for any $\Rr$-divisor $D$ on $X$. Let $\pi':X'\to Z'$ be the contraction defined by $-(K_{X'}+B_{\Phi_{i+1}}')$ over $Z$. By assumption, $\dim Z'=k$. Let $B_{Z'}^{(i+1)}$ and $\bM_{\pi'}$ be the discriminant and moduli parts of the canonical bundle formula for $(X',B_{\Phi_{i+1}}')$ over $Z'$ in Proposition \ref{prop:cbfindex} (respectively Proposition \ref{prop:bsa}) if $k=2$ (respectively $k=1$).

\begin{claim}\label{claim:lift2}
$p_i\bM_{\pi'}$ is base point free over $Z$, and 
$$p_{i}\left(K_{X'}+B_{\Phi_{i+1}}'\right)\sim p_{i}(\pi')^*\left(K_{Z'}+B_{Z'}^{(i+1)}+\bM_{\pi',Z'}\right).$$ 
\end{claim}
Assume Claim \ref{claim:lift2}. As $X\dashto X'$ is an MMP on $-(K_X+B_{\Phi_{i+1}})$ over $Z$, for any prime divisor $P$ on $X$ which is exceptional over $X'$, we have
$$a(P,X',B_{\Phi_{i+1}}')<a(P,X,B_{\Phi_{i+1}})\le1.$$
Thus we can find a crepant model $(\tilde{X},\tilde{B}^{(i+1)})$ of $(X',B_{\Phi_{i+1}}')$ such that $\tilde{X}$ and $X$ are isomorphic in codimension one.
\begin{center}
\begin{tikzcd}[column sep = 2em, row sep = 2em]
 X \arrow[d,"",swap]\arrow[rr, "", dashed]  && X' \arrow[d, "\pi'" swap]&&\tilde{X}\arrow[ll]\\
 Z && \arrow[ll, ""] Z'  && \arrow[ll, "\tilde{\tau}", swap]\tilde{Z}
\end{tikzcd}
\end{center}  
It is clear that if $(\tilde{X}/Z\ni z,\tilde{B}^{(i+1)})$ is $\cN_{i+1}$-complementary then so is $(X/Z\ni z,B_{\Phi_{i+1}})$. By Lemma \ref{lem:cbfcptoverbase}, we may find a crepant model $(\tilde{Z},B^{(i+1)}_{\tilde{Z}}+\bM_{\pi'})\to (Z',B^{(i+1)}_{Z'}+\bM_{\pi'})$ such that for any prime divisor $P\subseteq\Supp \tilde{B}^{(i+1)}$ which is vertical over $Z'$, the image of $P$ on $\tilde{Z}$ is a prime divisor. As $(X,B)$ is klt and $K_X+B\sim_{\Rr,Z}0$, we may find a boundary $B_{\tilde{Z}}\ge B_{\tilde{Z}}^{(i+1)}$ on $\tilde{Z}$ such that $(\tilde{Z},B_{\tilde{Z}}+\bM_{\pi'})$ is gklt and $K_{\tilde{Z}}+B_{\tilde{Z}}+\bM_{\pi',\tilde{Z}}\sim_{\Rr,Z}0.$ Since $p_i\bM_{\pi'}$ is base point free over $Z$, we can pick an effective $\Qq$-divisor $M_{\tilde{Z}}$ such that $p_iM_{\tilde{Z}}\sim_Z p_i\bM_{\pi',\tilde{Z}}$, $M_{\tilde{Z}}\wedge B_{\tilde{Z}}^{(i+1)}=0$, and $(\tilde{Z},B_{\tilde{Z}}+M_{\tilde{Z}})$ is klt. By our choice of $\cN_{i+1}$, $(\tilde{Z}/Z\ni z,B^{(i+1)}_{\tilde{Z}}+M_{\tilde{Z}})$ is $n$-complementary for some $n\in\cN_{i+1}$. By Proposition \ref{prop:cbflift1}, $(\tilde{X}/Z\ni z,\tilde{B}^{(i+1)})$ is also $n$-complementary. Therefore it suffices to prove Claim \ref{claim:lift2}.

\begin{proof}[Proof of Claim \ref{claim:lift2}]
 By Lemma \ref{lem:runantiMMP}, we may run an MMP on $-(K_{X'}+B_{\Phi_{i}}')\sim_{\Rr,Z'}B'_{\Phi_{i+1}}-B'_{\Phi_{i}}$ over $Z'$ and reach a good minimal model $X'\dashto X''$ over $Z'$, such that $B''_{\Phi_{i+1}}-B''_{\Phi_{i}}$ is semi-ample over $Z'$, where $D''$ denotes the strict transform of $D'$ on $X''$ for any $\Rr$-divisor $D'$ on $X'$. Let $\pi'':X''\to Z''$ be the contraction defined by $B''_{\Phi_{i+1}}-B''_{\Phi_{i}}$ over $Z'$, and $\tau:Z''\to Z'$ the induced morphism.

\begin{center}
	\begin{tikzcd}[column sep = 2em, row sep = 2em]
		X \arrow[d,"",swap]\arrow[rr, "", dashed]  && X' \arrow[d, "\pi'" swap] \arrow[rr, dashed] && X'' \arrow[d, "{\pi''}" swap] \\
		Z && \arrow[ll, ""] Z'  && \arrow[ll, "\tau", swap]Z'' 
	\end{tikzcd}
\end{center} 

We claim that $\tau$ is a birational morphism. In fact, one can pick a positive real number $\epsilon$, such that $B''-B_{\Phi_{i+1}}''+\epsilon (B_{\Phi_{i+1}}''-B_{\Phi_{i}}'')$ is semi-ample over $Z$ (see Lemma \ref{lem:relamp}) and that $X\dashto X''$ is also an MMP on $B-B_{\Phi_{i+1}}+\epsilon (B_{\Phi_{i+1}}-B_{\Phi_{i}})$ over $Z$. By assumption, 
$$\kappa(X/Z,B-B_{\Phi_{i+1}})=\kappa(X/Z,B-B_{\Phi_{i+1}}+\epsilon (B_{\Phi_{i+1}}-B_{\Phi_{i}})).$$ 
Hence we can see that $\tau:Z''\to Z'$ is birational.

By Lemma \ref{lem:cbfbirinv}, Proposition \ref{prop:cbfindex} and the choice of $p_i$, there exists a gklt g-pair $(Z'',B_{Z''}^{(i+1)}+\bM_{\pi'})$ induced by the canonical bundle formula, such that $p_{i}\bM_{\pi'}$ is base point free over $Z$ and
$$p_{i}\left(K_{X''}+B_{\Phi_{i+1}}''\right)\sim p_{i}(\pi'')^*\left(K_{Z''}+B_{Z''}^{(i+1)}+\bM_{\pi',Z''}\right).$$
Moreover, it is clear that
$$K_{Z''}+B_{Z''}^{(i+1)}+\bM_{\pi',Z''}=\tau^*\left(K_{Z'}+B_{Z'}^{(i+1)}+\bM_{\pi',Z'}\right).$$
Therefore
$$p_{i}(K_{X'}+B_{\Phi_{i+1}}')\sim p_{i}(\pi')^*\left(K_{Z'}+B_{Z'}^{(i+1)}+\bM_{\pi',Z'}\right)$$ 
as $(X',B_{\Phi_{i+1}}')$ is crepant to $(X'',B_{\Phi_{i+1}}'')$. We finish the proof.
\end{proof}
\end{proof}

\section{Proof of Theorem \ref{thm:lctype}}\label{sec7}
\subsection{Strictly lc Calabi-Yau pairs}
\begin{defn}\label{defn:cytype}
We say that $X$ is \emph{of Calabi-Yau type} over $Z$, if $X\to Z$ is a contraction, and there is a boundary $C$ such that $(X,C)$ is klt and $K_X+C\sim_{\Rr,Z}0$.
\end{defn}
\begin{lem}\label{lem:cyt}
Suppose that $X$ is of Calabi-Yau type over $Z$. Assume that $(X,B^+)$ is lc, $K_X+B^+\sim_{\Rr,Z}0$ for some boundary $B^+$, and $f: Y \to X$ is a projective birational morphism from a normal quasi-projective variety $Y$, such that $a(E_{i},X,B^+)<1$ for any prime exceptional divisor $E_{i}$ of $f$. Then $Y$ is of Calabi-Yau type over $Z$.
\end{lem}
\begin{proof}
Since $X$ is of Calabi-Yau type over $Z$, there exists a klt pair $(X, C)$ such that $K_X+C\sim_{\Rr,Z}0$. Let $D_{t}:=t B^{+}+(1-t)C$. Then $(X,D_{t})$ is klt and $K_{X}+D_{t}\sim_{\Rr,Z}0$ for any $t \in[0,1)$. We have
$$K_{Y}+B_{Y}^{+}+\sum_{i} (1-a_{i}^{+} ) E_{i}=f^{*} (K_{X}+B^{+} )$$
and
$$K_{Y}+C_{Y}+\sum_{i} (1-a_{i} ) E_{i}=f^{*} (K_{X}+C ),$$
where $B_{Y}^{+}$ and $C_{Y}$ are the strict transforms of $B^{+}$ and $C$ on $Y$ respectively, $a_{i}^{+}:=a (E_{i}, X, B^{+} )<1$, and $a_{i}:=a (E_{i}, X, C )$ for any $i$. Then
$$
\begin{aligned}
& K_{Y}+D_{t, Y}:=f^{*} (K_{X}+D_{t} ) \\
=& K_{Y}+t B_{Y}^{+}+(1-t) C_{Y}+\sum_{i} (t (1-a_{i}^{+} )+(1-t) (1-a_{i} ) ) E_{i} .
\end{aligned}
$$
Pick $0<t_{0}<1$ such that $t_{0} (1-a_{i}^{+} )+ (1-t_{0} ) (1-a_{i} ) \geq 0$ for any $i$. Then $(Y,D_{t_0,Y})$ is klt and $K_Y+D_{t_0,Y}\sim_{\Rr,Z}0$. So $Y$ is of Calabi-Yau type over $Z$.
\end{proof}

\begin{defn}[cf. {\cite[\S11]{Sho20}}]
A pair $(X/Z\ni z,B)$ is called \emph{strictly lc Calabi-Yau} if
\begin{enumerate}
  \item $(X/Z\ni z,B)$ is an $\Rr$-complement of itself, and
  \item for any $\Rr$-complement $(X/Z\ni z,B^+)$ of $(X/Z\ni z,B)$, $B^+=B$ over some neighborhood of $z$.
\end{enumerate}
\end{defn}

\begin{rem}
When $\dim Z=0$, $(X,B)$ is strictly lc Calabi-Yau if and only if $(1)$ holds.
\end{rem}

\begin{lem}\label{lem:mlcislccenter}
Assume that $(X/Z\ni z,B)$ is an $\Rr$-complement of itself. Then $(X/Z\ni z,B)$ is strictly lc Calabi-Yau if and only if either $\dim z=\dim Z$ or $\bar{z}$ is the image of an lc center of $(X, B)$ on $Z$.
\end{lem}
\begin{proof}
First assume that $(X/Z\ni z,B)$ is strictly lc Calabi-Yau. Suppose that $\dim z<\dim Z$ and $\bar{z}$ is not the image of any lc center of $(X,B)$. Possibly shrinking $Z$ near $z$, we may find an ample divisor $H\ge0$ such that $z\in\Supp H$ and $H$ does not contain the image of any lc center of $(X,B)$. Pick a positive real number $\epsilon$, such that $(X/Z\ni z,B+\epsilon\pi^*H)$ is lc and thus an $\Rr$-complement of $(X/Z\ni z,B)$. However, $B+\epsilon\pi^*H\neq B$ over any neighborhood of $z$, a contradiction. 

Now we prove the converse direction. Assume that $(X/Z\ni z,B+G)$ is an $\Rr$-complement of $(X/Z\ni z,B)$ for some $G\ge0$. Since $G\sim_\Rr0$ over some neighborhood of $z$, $G=\pi^*L_Z$ for some $\Rr$-Cartier $\Rr$-divisor $L_Z$ on $Z$ by \cite[Lemma 2.4]{CHL22}. If $\dim z=\dim Z$, then $G=0$ over a neighborhood of $z$. If $\bar{z}$ is the image of some lc center of $(X,B)$, then $z\notin \Supp L_Z$ as $(X,B+G)$ is lc over a neighborhood of $z$. Therefore in both cases, $(X/Z\ni z,B)$ is strictly lc Calabi-Yau.
\end{proof}

\begin{ex}
Let $\pi:X:=\Pp^1\times\Pp^1\to Z:=\Pp^1$, $z\in Z$ a closed point, and $L_1,L_2$ two sections. Then over a neighborhood of $z$, we have $(X,L_1+L_2)$ is lc and $K_X+L_1+L_2\sim_{\Rr,Z}0$. Since $K_X+L_1+L_2+\pi^*z\sim_{\Rr,Z}0$, $(X/Z\ni z,L_1+L_2)$ is not strictly lc Calabi-Yau.
\end{ex}

\begin{lem}\label{lem:crepantofmlccy}
Suppose that $(X/Z\ni z,B)$ is strictly lc Calabi-Yau and $X\dashrightarrow X'$ is a birational contraction over $Z$. Let $B'$ be the strict transform of $B$ on $X'$. Then $(X'/Z\ni z,B')$ is strictly lc Calabi-Yau.
\end{lem}
\begin{proof}
It follows from the definition of strictly lc Calabi-Yau and the fact that $(X,B)\dashrightarrow (X',B')$ is crepant over some neighborhood of $z$.
\end{proof}

\begin{prop}\label{prop:maxlcindex}
Let $\Gamma\subseteq[0,1]\cap\Qq$ be a DCC set. Then there exists a positive integer $I$ depending only on $\Gamma$ satisfying the following. If $(X/Z\ni z,B)$ is a strictly lc Calabi-Yau threefold pair such that $B\in\Gamma$, then $I(K_X+B)\sim0$ over some neighborhood of $z$.
\end{prop}
\begin{proof}
Possibly replacing $(X/Z\ni z,B)$ by a $\Qq$-factorial dlt modification, we may assume that $X$ is $\Qq$-factorial. Since $(X/Z\ni z,B)$ is a strictly lc Calabi-Yau pair, by \cite[Theorem 5.20]{HLS19}, $B\in\Gamma'$ over a neighborhood of $z$ for some finite subset $\Gamma'\subseteq\Gamma$ which only depends on $\Gamma$. According to \cite[Theorem 2.12]{CHL22}, we may find a positive integer $I$ which only depends on $\Gamma'$ such that $(X/Z\ni z,B)$ has a monotonic $I$-complement $(X/Z\ni z,B+G)$ for some $G\ge 0$. By assumption, $G=0$ over some neighborhood of $z$. Thus
$$I(K_X+B)=I(K_X+B+G)\sim0$$
over some neighborhood of $z$. 
\end{proof}

\subsection{Proof of Theorem \ref{thm:lctype}}
We first show a special case of Theorem \ref{thm:lctype}. 

For convenience, we say a pair $(X,B)$ is klt over a closed subset $Z_0\subseteq Z$, if $a(E,X,B)>0$ for any prime divisor $E$ over $X$ such that $\pi(\Center_X(E))\subseteq Z_0$, where $\pi:X\to Z$ is a contraction. For two $\Rr$-divisors $D_1$ and $D_2$ on $X$, by $D_1\ge D_2$ (respectively $D_1> D_2$) over $Z_0$, we mean that $\mult_ED_1\ge\mult_ED_2$ (respectively $\mult_ED_1>\mult_ED_2$) for any prime divisor $E$ on $X$ with $\pi(E)\subseteq Z_0$. By $D_1\ge D_2$ over an open subset $U\subseteq Z$, we mean $D_1|_{\pi^{-1}(U)}\ge D_2|_{\pi^{-1}(U)}$.

\begin{prop}\label{prop:splctype}
Let $I$ be a positive integer. Assume that $\cN$ is a finite set of positive integers divisible by $I$ given by Theorem \ref{thm:3foldsbddcmpt} which only depends on $I$.

Assume that $(X/Z\ni z,B)$ is an $\Rr$-complementary threefold pair such that $X$ is of Calabi-Yau type over $Z$. Assume that there is a contraction $\pi':X\to Z'$ over $Z$, and an open subset $U\subseteq Z'$, such that 
\begin{enumerate}
  \item $IB\in\Zz_{\ge0}$ over $U$, 
  \item $(X,B)$ is klt over $Z'\setminus U$, and
  \item $-(K_X+B)\sim_{\Rr}(\pi')^*H'$ for some $\Rr$-divisor $H'$ which is ample over $Z$.
\end{enumerate} 
Then $(X/Z\ni z,B)$ is $\cN$-complementary.
\end{prop}
\begin{proof}
Possibly shrinking $Z$ near $z$, we may assume that $(X,B)$ is lc. Set $N:=\max_{n\in\cN}{n}$. We claim that there exists a boundary $B'$ on $X$ such that 
\begin{itemize}
 \item $(X,B')$ is klt, 
 \item $K_X+B'\sim_{\Rr,Z}0$, and
 \item $B'\ge \frac{N}{N+1} B$ over $U$ and $B'\ge B$ over $Z'\setminus U$.
\end{itemize}
Assume the claim holds. Then
\begin{equation}\label{eqn: prop 7.9}
    \lf(n+1)B'\rf\ge n\lf B\rf+\lf(n+1)\{B\}\rf
\end{equation}
for any $n\in\cN$. By Theorem \ref{thm:3foldsbddcmpt} and the construction of $\cN$, $(X/Z\ni z,B')$ is $n$-complementary for some $n\in\cN$. Thus $(X/Z\ni z,B)$ is $n$-complementary by \eqref{eqn: prop 7.9}.

Therefore it suffices to prove the claim. By assumption, we may find an effective $\Rr$-Cartier $\Rr$-divisor $H_1'\sim_{\Rr,Z}H'$ such that $Z'\setminus U\subseteq\Supp H_1'$ and $(X,B+H_1)$ is lc, where $H_1:=(\pi')^*H_1'$. In particular, we have 
$$K_X+B+H_1\sim_{\Rr,Z}0,\text{ and }B+H_1>B\text{ over }Z'\setminus U.$$ 
Since $X$ is of Calabi-Yau type over $Z$, there exists a boundary $C$ such that $(X,C)$ is klt and $K_X+C\sim_{\Rr,Z}0$. Let $\delta\in(0,1)$ be a positive real number such that
$$B':=(1-\delta)(B+H_1)+\delta C\ge \frac{N}{N+1}B\text{ over }U\text{, and }(1-\delta)(B+H_1)\ge B\text{ over }Z'\setminus U.$$
It is clear that $(X,B')$ is klt and $K_X+B'\sim_{\Rr,Z}0$. This completes the proof.
\end{proof}

\begin{proof}[Proof of Theorem \ref{thm:lctype}]
Let $I=I(l, \Gamma\cap\Qq)$ be a positive integer divisible by $l$ given by Proposition \ref{prop:maxlcindex} which only depends on $l$ and $\Gamma\cap\Qq$, and let $\Phi:=\Phi(\frac{1}{I}\Zz\cap[0,1])$. Let $\cN=\cN(I)$ be a finite set of positive integers divisible by $I$ given by Theorem \ref{thm:3foldsbddcmpt} which only depends on $I$. We will show that $\cN$ has the required property.

Possibly shrinking $Z$ near $z$, we may assume that $(X,B)$ is lc. By Lemma \ref{lem:cyt}, we can replace $(X,B)$ by a dlt modification and thus assume that $(X,B)$ is $\Qq$-factorial dlt. Suppose that $(X/Z\ni z,B^+)$ is an $\Rr$-complement of $(X/Z\ni z,B)$. Possibly replacing $z$ by a closed point of $\bar{z}$ and shrinking $Z$ near $z$, we may assume that $z$ is a closed point, $(X,B^+)$ is lc, and $K_X+B^+\sim_{\Rr,Z}0$. Write
$$-(K_X+B_{\cN\_\Phi})\sim_{\Rr,Z} B^+-B_{\cN\_\Phi}=F+M,$$ 
where $F:=N_\sigma(B^+-B_{\cN\_\Phi}/Z)\ge0$ and $M:=B^+-B_{\cN\_\Phi}-F\ge0$
(cf. \cite[III, \S4]{Nak04}, \cite[\S3]{LX22}). Note that $F$ is well-defined as $B^+-B_{\cN\_\Phi}\ge0$. Since $X$ is of Calabi-Yau type over $Z$, there exists a boundary $C$ such that $(X,C)$ is klt and $K_X+C\sim_{\Rr,Z}0$. Choose a positive real number $\epsilon_0$ such that $(X,C+\epsilon_0 M)$ is klt. We may run an MMP on $K_X+C+\epsilon_0M$ over $Z$ and reach a good minimal model $X'$, such that $K_{X'}+C'+\epsilon_0M'$ is semi-ample over $Z$, where $D'$ denotes the strict transform of $D$ on $X'$ for any $\Rr$-divisor $D$ on $X$. Since
$$K_X+C+\epsilon_0M\sim_{\Rr,Z}\epsilon_0M\sim_{\Rr,Z}-\epsilon_0(K_X+B_{\cN\_\Phi}+F),$$
$X$ and $X'$ are isomorphic in codimension one by \cite[Lemma 2.4]{HX13}. We also see that $-(K_{X'}+B'_{\cN\_\Phi}+F')$ is semi-ample over $Z$ and thus induces a contraction $\pi':X'\to Z'$ over $Z$. In particular, there is an effective $\Rr$-divisor $H'$ on $Z'$ which is ample over $Z$ such that $-(K_{X'}+B'_{\cN\_\Phi}+F')\sim_{\Rr}(\pi')^*H'$. Note that $(X',B^{+\prime})$ is lc and thus $(X',B_{\cN\_\Phi}'+F')$ is also lc.

Let $\eta'$ be the generic point of $Z'$, and
$$\cZ_{scy}:=\{\eta'\} \cup \{z'\in Z'\mid (X'/Z'\ni z',B'_{\cN\_\Phi}+F')\text{ is strictly lc Calabi-Yau}\}.$$
By Lemma \ref{lem:mlcislccenter} and \cite[Theorem 1.1]{AM06}, $\cZ_{scy}$ is a non-empty finite set.

\begin{claim}\label{claim:twofacts}
We have $I(K_{X'}+B'_{\cN\_\Phi}+F')\sim0$ over a neighborhood of $z'$ for any $z'\in\cZ_{scy}$.
\end{claim}
Assume Claim \ref{claim:twofacts}. By \cite[Lemma 2.5]{CHL22}, Lemma \ref{lem:mlcislccenter} and \cite[Theorem 1.1]{AM06}, there exists an open subset $U\subseteq Z'$ such that 
$$I\left(K_{X'}+B'_{\cN\_\Phi}+F'\right)\sim0\text{ over $U$ and }\left(X',B'_{\cN\_\Phi}+F'\right)\text{ is klt over }Z'\setminus U.$$ 
In particular, $I(B'_{\cN\_\Phi}+F')\in\Zz_{\ge0}$ over $U$. Recall that $-(K_{X'}+B'_{\cN\_\Phi}+F')\sim_{\Rr}(\pi')^*H'$ where $H'$ is ample over $Z$. By Proposition \ref{prop:splctype}, $(X'/Z\ni z,B'_{\cN\_\Phi}+F')$ is $\cN$-complementary, and thus $(X'/Z\ni z,B')$ is also $\cN$-complementary by Lemma \ref{lem:N_Phicompl}. Moreover, as $X$ and $X'$ are isomorphic in codimension one, $(X/Z\ni z,B)$ is $\cN$-complementary.
\end{proof}

\begin{proof}[Proof of Claim \ref{claim:twofacts}]
We may pick a positive real number $\epsilon$ such that $(X',C'+\epsilon F')$ is klt and that $X\dashrightarrow X'$ is a sequence of steps of the $-(K_X+B_{\cN\_\Phi}+(1-\epsilon)F)$-MMP over $Z$. Since
\begin{align*}
-\left(K_{X'}+B'_{\cN\_\Phi}+(1-\epsilon)F'\right)&\sim_{\Rr,Z'}\epsilon F'\sim_{\Rr,Z'}K_{X'}+C'+\epsilon F',
\end{align*}
one can run an MMP on $-(K_{X'}+B'_{\cN\_\Phi}+(1-\epsilon)F')$ over $Z'$. This MMP terminates with a model $X''$ on which $-(K_{X''}+B''_{\cN\_\Phi}+(1-\epsilon)F'')$ is semi-ample over $Z'$, where $D''$ denotes the strict transform of $D'$ on $X''$ for any $\Rr$-divisor $D'$ on $X'$. 
\begin{center}
	\begin{tikzcd}[column sep = 2em, row sep = 2em]
		X \arrow[d, "" swap]\arrow[rr, dashed]  && X' \arrow[d, "" swap] \arrow[rr,  "\psi'", dashed] && X''  \\
		Z &&  \arrow[ll, ""]  Z'  && 
	\end{tikzcd}
\end{center}
For any positive real number $\epsilon'\ll\epsilon$, we infer that $\psi: X\dashrightarrow X''$ is also an MMP on $-(K_{X}+B_{\cN\_\Phi}+(1-\epsilon')F)$ over $Z$, and
\begin{align*}
	&-\left(K_{X''}+B''_{\cN\_\Phi}+\left(1-\epsilon'\right)F''\right)\\=&-(1-\frac{\epsilon'}{\epsilon})\left(K_{X''}+B''_{\cN\_\Phi}+F''\right)-\frac{\epsilon'}{\epsilon}\left(K_{X''}+B''_{\cN\_\Phi}+(1-\epsilon)F''\right)
\end{align*}
is semi-ample over $Z$ (see Lemma \ref{lem:relamp}). In particular, $X''$ is a good minimal model of $-(K_{X}+B_{\cN\_\Phi}+(1-\epsilon')F)$ over $Z$. Since $N_\sigma(-(K_{X}+B_{\cN\_\Phi}+(1-\epsilon')F)/Z)=\epsilon' F$ (cf. \cite[III, 4.2 Lemma]{Nak04}), by \cite[Lemma 2.4]{HX13}, $F$ is contracted by $\psi$. Hence
$$K_{X''}+B''_{\cN\_\Phi}=\psi'_*\left(K_{X'}+B'_{\cN\_\Phi}+F'\right).$$
If $\dim X'=\dim Z'$ and $z'=\eta'\in\cZ_{scy}$, then $X'$ is smooth and $B^{+\prime}=0$ over a neighborhood of $\eta'$. Otherwise, by Lemma \ref{lem:crepantofmlccy}, we know that $({X''}/Z'\ni z',B''_{\cN\_\Phi})$ is strictly lc Calabi-Yau for any $z'\in\cZ_{scy}$, which implies that $B''_{\cN\_\Phi}=B^{+\prime\prime}=B''\in\Gamma\cap\Qq$ over a neighborhood of $z'$. We therefore see that $I(K_{X''}+B''_{\cN\_\Phi})\sim0$ over some neighborhood of $z'$ by our choice of $I$. Since $(X',B'_{\cN\_\Phi}+F')$ is crepant to $(X'',B''_{\cN\_\Phi})$, our claim holds.
\end{proof}

\bibliographystyle{alpha}

\end{document}